\documentclass[11pt]{amsart}

\pdfoutput=1

\usepackage[text={420pt,660pt},centering]{geometry}

\usepackage{cancel}
\usepackage{esint,amssymb} 
\usepackage{graphicx}
\usepackage{MnSymbol}
\usepackage{mathtools} 
\usepackage[colorlinks=true, pdfstartview=FitV, linkcolor=blue, citecolor=blue, urlcolor=blue,pagebackref=false]{hyperref}
\usepackage{microtype}

\usepackage{bm}
\usepackage{dsfont}
\usepackage{mathrsfs}
\usepackage{xcolor}

\usepackage[normalem]{ulem}

\newtheorem{proposition}{Proposition}
\newtheorem{theorem}[proposition]{Theorem}
\newtheorem{lemma}[proposition]{Lemma}
\newtheorem{corollary}[proposition]{Corollary}

\theoremstyle{remark}
\newtheorem{remark}[proposition]{Remark}

\theoremstyle{definition}
\newtheorem{definition}[proposition]{Definition}

\numberwithin{equation}{section}
\numberwithin{proposition}{section}
\numberwithin{figure}{section}
\numberwithin{table}{section}

\newcommand{\N}{\mathbb{N}}

\newcommand{\R}{\mathbb{R}}

\newcommand{\E}{\mathbb{E}}

\newcommand{\Rd}{{\mathbb{R}^d}}

\renewcommand{\S}{\mathbf{S}}

\newcommand{\ep}{\varepsilon}
\newcommand{\eps}{\varepsilon}

\renewcommand{\le}{\leqslant}
\renewcommand{\ge}{\geqslant}
\renewcommand{\leq}{\leqslant}
\renewcommand{\geq}{\geqslant}

\renewcommand{\subset}{\subseteq}
\renewcommand{\bar}{\overline}

\newcommand{\Ll}{\left}
\newcommand{\Rr}{\right}
\renewcommand{\d}{\mathrm{d}}
\newcommand{\dr}{\partial}

\newcommand{\msf}{\mathsf}

\newcommand{\al}{\alpha}

\DeclareMathOperator{\tr}{tr}

\newcommand{\la}{\left\langle}
\newcommand{\ra}{\right\rangle}

\renewcommand{\H}{\mathsf{H}}

\newcommand{\diag}{\mathfrak{d}}

\begin{document}

\author[H.-B. Chen]{Hong-Bin Chen}
\address[H.-B. Chen]{Courant Institute of Mathematical Sciences, New York University, New York, New York, USA}
\email{hbchen@cims.nyu.edu}

\author[J.-C. Mourrat]{Jean-Christophe Mourrat}
\address[J.-C. Mourrat]{Courant Institute of Mathematical Sciences, New York University, New York, New York, USA; ENS Lyon and CNRS, Lyon, France}
\email{jean-christophe.mourrat@ens-lyon.fr}

\author[J. Xia]{Jiaming Xia}
\address[J. Xia]{Department of Mathematics, University of Pennsylvania, Philadelphia, Pennsylvania, USA}
\email{xiajiam@sas.upenn.edu}

\subjclass[2010]{}
\date{\today}

\title{Statistical inference of finite-rank tensors}

\begin{abstract}
We consider a general statistical inference model of finite-rank tensor products. For any interaction structure and any order of tensor products, we identify the limit free energy of the model in terms of a variational formula. Our approach consists of showing first that the limit free energy must be the viscosity solution to a certain Hamilton-Jacobi equation.
\end{abstract}

\maketitle


%
%
%
%
%
%

\section{Introduction} 

\subsection{Setting}
\label{s.setting}
Let $K,L,\mathsf{p}\in\N$ and $A\in \R^{K^\mathsf{p}\times L}$, which will be kept fixed throughout the paper. For every $N\in \N$, $t\geq 0$ and a random matrix $X\in \R^{N\times K}$, we consider the inference task of recovering $X$ from the observation of
\begin{align}\label{e.Y}
    Y := \sqrt{\frac{2t}{N^{\mathsf{p}-1}}}X^{\otimes \mathsf{p}}A+W \  \in \R^{N^\mathsf{p}\times L},
\end{align}
where $\otimes$ denotes the tensor product of matrices, and $W\in \R^{N^\mathsf{p}\times L}$, {independent of the randomness of $X$}, consists of independent standard Gaussian entries (we view $X^{\otimes \mathsf{p}}$ as an $N^\mathsf{p}$-by-$K^\mathsf{p}$ matrix). Throughout, the dot product between two vectors or matrices of the same size is the entry-wise inner product. The associated norm is denoted by $|\cdot|$. For convenience of analysis, we assume that the random matrix $X$ almost surely satisfies 
\begin{align}\label{e.supp}
    |X|\leq \sqrt{NK}.
\end{align}
{For instance, \eqref{e.supp} is satisfied if every entry of $X$ has its absolute value bounded by $1$.}
We denote the law of $X$ by $P_N^X$.
Using Bayes' rule, the law of $X$ conditioned on observing~$Y$ is the  measure proportional to $e^{H^\circ_N(t,x)} \, \d P^X_N(x)$, where the Hamiltonian $H^\circ_N$ is 
\begin{equation*}
H^\circ_N(t,x) := \sqrt{\frac{2t}{N^{\mathsf{p}-1}}}(x^{\otimes \mathsf{p}}A)\cdot Y-\frac{t}{N^{\mathsf{p}-1}}|x^{\otimes \mathsf{p}}A|^2.
\end{equation*}
The associated free energy is given by 
\begin{equation*}
    F^\circ_N(t) := \frac{1}{N} \log \int_{\R^{N\times K}}e^{H^\circ_N(t,x)}\, \d P^X_N(x).
\end{equation*}
The mutual information $I(X,Y)$ between $X$ and $Y$ is an important information-theoretical quantity, which is equal to $\E F^\circ_N(t)$ up to a simple additive term. Computing the limit of the mutual information as $N\to\infty$ allows one to determine the critical value of $t$ below which the inference task is theoretically impossible. Therefore, the limit of $\E F^\circ_N(t)$ is the central object of investigation in many inference models. For more details, we refer to the discussion in \cite{barbier2016}.

In order to analyze this model, we start by enriching the system by adding an additional observation $\bar Y=X\sqrt{2h}+Z$ for $h\in\S^K_+$, where $\S^K_+$ is the set of $K\times K$ symmetric positive semi-definite matrices, and {$Z\in \R^{N\times K}$, independent of all other sources of randomness previously introduced,} consists of i.i.d.\ standard Gaussian entries. Then, the law of $X$ conditioned on observing $Y$ and $\bar Y$ is a Gibbs measure {proportional to $e^{H_N(t,h,x)}\d P_N^X(x)$} with Hamiltonian
\begin{align*}
    H_N(t,h,x) := H^\circ_N(t,x)+ \sqrt{2h}\cdot(x^\intercal \bar Y)-h\cdot(x^\intercal x).
\end{align*}
The corresponding free energy is 
\begin{align}\label{e.F_N}
    F_N(t,h) := \frac{1}{N}\log\int_{\R^{N\times K}} e^{H_N(t,h,x)} \, \d P^X_N(x).
\end{align}
We also set $\bar F_N = \E F_N$. Note that the initial free energy satisfies $F^\circ_N(t) = F_N(t,0)$. 
We let $\H:\S^K_+\to \R$ be the mapping such that, for every $q \in \S^K_+$,
\begin{align}\label{def.H}
    \H(q) := (AA^\intercal)\cdot q^{\otimes \mathsf{p}}.
\end{align}
Our main result is the identification of the limit free energy, for any given choice of interaction matrix $A$ and $\mathsf{p} \in \N$. 
\begin{theorem}\label{main_thm}
In addition to \eqref{e.supp}, suppose that
\begin{itemize}
    \item $\big(\bar F_N(0,\cdot)\big)_{N\in\N}$ converges pointwise to some $C^1$ function $\psi:\S^K_+\to \R$;
    \item $\lim_{N\to \infty}\E\|F_N-\bar F_N\|^2_{L^\infty(D)}=0$ for every compact $D\subset [0,\infty)\times \S^K_+$.
\end{itemize}
Then, for every $(t,h)\in[0,\infty)\times \S^K_+$, we have
\begin{align}
\label{e.main_thm}
    \lim_{N\to\infty}\bar F_N(t,h) = \sup_{h''\in\S^K_+}\inf_{h'\in\S^K_+}\big\{h''\cdot (h-h')+\psi(h')+t\H(h'')\big\}.
\end{align}
\end{theorem}

\begin{remark}
The above convergence can be improved into convergence in the local uniform topology by using that $\bar F_N$ is Lipschitz uniformly over $N$ (see Lemma~\ref{l.lip_unif_N}). 
\end{remark}

We briefly comment on the hypotheses of the theorem. One can see that $F_N(0,\cdot)$ is the free energy associated with a decoupled system where the only observation $\bar Y$ is linear in~$X$. Therefore, in many cases, the limit of $\bar F_N(0,\cdot)$ can be computed straightforwardly. In particular, if $P_N^X$ is the $N$-fold tensor product of a fixed probability measure on $\R^K$, then $\bar F_N(0,\cdot)$ in fact does not depend on $N$, and is $C^1$. The next assumption can be rephrased as local uniform concentration of $F_N$. Again, this condition is straightforward to verify in many models, with standard tools available: see for instance \cite[Lemma~C.1]{HBJ} for the case when the rows of $X$ are i.i.d.\ and bounded. 

Among our assumptions, perhaps the only surprising one is the requirement that $\psi$ be of class $C^1$. For certain choices of the nonlinearity $\H$, such as when $\H$ is convex, this assumption is not necessary (see for instance \cite{HBJ}). However, when considering arbitrary choices of $A$ and $\mathsf{p}$ as we do here, this assumption may be required. In a simpler setting, we illustrate the usefulness of this assumption in Remark~\ref{r.diff.assumption}. 

\subsection{Related works}
Many inference models can be viewed as special cases of \eqref{e.Y}. Indeed, one could argue that essentially any ``fully-connected'' inference problem will have the form of \eqref{e.Y} for some suitable choice of $A$ and $\mathsf{p}$. Among them, the models where the limit free energy has been studied include the spiked Wigner model \cite{barbier2016, lelarge2019fundamental, barbier2019adaptive, mourrat2018hamilton, mourrat2019hamilton}, the spiked Wishart model \cite{miolane2017fundamental, barbier2017layered, kadmon2018statistical, luneau2020high, chen2020hamilton}, the stochastic block model (or community detection problem) \cite{lelarge2019fundamental, mayya2019mutualIEEE, reeves2019geometryIEEE}, 
the inference of second order matrix tensor products~\cite{reeves2020information}, and the inference of higher order vector tensor products \cite{lesieur2017statistical, barbier2019adaptive, mourrat2018hamilton}.
The model closest to~\eqref{e.Y} is the inference of finite-rank even-order tensor products studied in \cite{luneau2019mutual}. The case of tensors of odd order was left open there, see \cite[Section~7]{luneau2019mutual}. In Section~\ref{s.app}, we apply our main result to this model, for tensor products of arbitrary order ($\mathsf{p} \in \N$). For a more detailed discussion on these models, we refer to the introduction in~\cite{HBJ}.

Many of the results mentioned above were obtained by the powerful method of adaptive interpolation introduced in \cite{barbier2019adaptive, barbier2019adaptive2} and refined in subsequent works. In \cite{reeves2020information}, a novel extension using interpolation paths parameterized by order-preserving positive semi-definite matrices was employed to completely describe the limit in the general second order tensor products model. The order-preserving property (\cite[Proposition 4]{reeves2020information}) has a similar counterpart that plays a crucial role in this work (Lemma~\ref{l.F_N.order} and Proposition~\ref{p.weak-tensor}).

The approach taken up in the present paper is based instead on identifying the limit free energy as the viscosity solution to a certain Hamilton-Jacobi equation. This alternative approach was introduced in \cite{mourrat2018hamilton, mourrat2019hamilton}, and can also inform the analysis of spin glass models \cite{mourrat2019parisi, mourrat2020extending, mourrat2020nonconvex, mourrat2020free};  related considerations also appeared in the physics literature \cite{genovese2009mechanical, guerra2001sum,  barra2010replica, barra2013mean}. 


The setting of the present paper is identical to that of \cite{HBJ}, in which partial results were obtained. There, for general interaction matrix $A$ and order $\mathsf{p}$, only an upper bound on the limit free energy could be proved; a complete identification of this limit could only be obtained for particular choices of $A$ and $\mathsf{p}$. Here, we close this gap and cover all cases in a unified approach. 

Compared with \cite{HBJ}, the main novelty of the present paper is that we will rely on a different method for the identification of the viscosity solution. This method relies crucially on the fact that the functions under consideration are \textit{convex}. We explain this new uniqueness criterion in the simpler context of Hamilton-Jacobi equations on $[0,\infty) \times \R^d$ in the appendix. The gist of our work is then to extend this criterion to Hamilton-Jacobi equations posed on $[0,\infty) \times \S^K_+$, and then to verify that any possible limit of the free energy does satisfy this criterion.

The rest of the paper is organized as follows. In Section~\ref{s.prop.F_N}, we present basic properties of $\bar F_N$. In particular, we record that $\bar F_N$ is convex, nondecreasing, and has nondecreasing gradients. In Section~\ref{s.convex}, we recall basic facts of convex analysis and prove some useful results in preparation for the study of the Hamilton-Jacobi equation. Using these, we prove a convenient criterion for identifying viscosity solutions in Section~\ref{s.hj}. Lastly, Section~\ref{s.proof} contains the proof of Theorem~\ref{main_thm} and an application to the model \eqref{e.special}.

\subsection*{Acknowledgements}
We would like to warmly thank Stefano Bianchini for providing us with the idea for the proof of Proposition~\ref{p.weak}. JCM was partially supported by the NSF grant DMS-1954357.

\section{Properties of the free energy}\label{s.prop.F_N}

In this section, we study basic properties of $\bar F_N$. We start by introducing notation.

For any measurable $g:\R^{N\times K}\to \R^m$ {for some $m\in\N$}, we denote by $\la g(x)\ra$ the expectation of $g$, coordinatewise, with respect to the Gibbs measure proportional to $e^{H_N(t,h,x)} \, \d P^X_N(x)$, which can also be written as $\la g(x)\ra = \E [g(X)|Y,\bar Y]$ for $Y$ and $\bar Y$ introduced in the previous section. Note that the dependence of $\la\, \cdot \,\ra$ on $t,h$ is suppressed from the notation when there is no confusion. Within the bracket $\la\, \cdot \,\ra$, we denote by $x',x'',x'''$ independent copies of $x$, which are called replicas of $x$. The transpose operator on matrices is denoted by superscript $\intercal$.

In addition to $\S^K_+$, we denote by $\S^K$ and $\S^K_{++}$, the set of $K\times K$ symmetric matrices, and symmetric positive definite matrices, respectively. We view $\S^K$ as an ambient linear space for $\S^K_+$ and $\S^K_{++}$. By choosing an orthonormal basis with respect to the entry-wise dot product, we can identify $\S^K$ with $\R^{K(K+1)/2}$ isometrically. Therefore, differentiation makes sense on $\S^K$ as the usual one on Euclidean spaces. Naturally, we also identify the dual space of $\S^K$ with itself.
For a function $g:[0,\infty)\times \S^K_+\to\R$ which is differentiable at $(t',h')$, we denote by $\partial_t g(t',h')\in\R$ its derivative with respect to the first variable, and by $\nabla g(t',h')\in\S^K$ the gradient with respect to the second variable.

Using the expression \eqref{e.F_N}, we can compute that
\begin{gather}  \label{e.dt_F_N}
\dr_t \bar F_N = \frac{1}{N^{\mathsf{p}}} \E \la x^{\otimes \mathsf{p}}A \cdot x'^{\otimes\mathsf{p}} A \ra=\frac{1}{N^{\mathsf{p}}} \E\left[ \la x^{\otimes \mathsf{p}}A \ra \cdot \la x^{\otimes\mathsf{p}} A \ra\right],
\\
\label{e.grad_F_N}
\nabla \bar F_N = \frac 1 N \E \la x^\intercal x' \ra = \frac{1}{N}\E\left[\la x \ra^\intercal \la x \ra\right].
\end{gather}
This computation involves the Nishimori identity, the Gaussian integration by parts, and the independence of replicas with respect to the Gibbs measure. For details, we refer to \cite[(3.5)-(3.6)]{HBJ}.
Recalling the definition of $\H$ in \eqref{def.H}, we obtain that~$\bar F_N$ satisfies
\begin{align*}
    \partial_t \bar F_N - \H\big(\nabla \bar F_N\big) = \frac{1}{N^\mathsf{p}}\bigg(\E\big\langle\H(x^\intercal x') \big\rangle - \H\big(\E\langle x^\intercal x'\rangle \big)\bigg),
\end{align*}
and the right-hand side is expected to be small when $N$ is large. Hence, $\bar F_N$ can be viewed to approximately satisfy the Hamilton-Jacobi equation
\begin{align}\label{e.HJ_tensor}
    \partial_t f - \H(\nabla f)=0 \quad\text{in}\ [0,\infty)\times \S^K_+.
\end{align}
This is the key insight for the Hamilton-Jacobi equation approach. Later, we will show that indeed $\bar F_N$ converges to the unique solution to \eqref{e.HJ_tensor}; and then that this solution admits the variational representation appearing on the right side of \eqref{e.main_thm}.

In the remaining two subsections, we collect useful properties of derivatives of $\bar F_N$ and prove that $\bar F_N$ is convex.

\subsection{Derivatives of free energy}

We record basic results on the derivatives of $\bar F_N$.
\begin{lemma}\label{l.lip_unif_N}
For each $N\in\N$, the function $\bar F_N$ is $C^1$ and the following holds:
\begin{gather*}
    \sup_{N\in\N,\, (t,h)\in [0,\infty)\times \S^K_+}\big|(\partial_t, \nabla)\bar F_N\big|(t,h)<\infty;\\
    (\partial_t, \nabla)\bar F_N(t,h)\in [0,\infty)\times \S^K_+,\quad\forall N\in\N,\ (t,h)\in [0,\infty)\times \S^K_+.
\end{gather*}
\end{lemma}

\begin{proof}
{It follows from \eqref{e.dt_F_N} and \eqref{e.grad_F_N}, along with the assumption~\eqref{e.supp}.}
\end{proof}

The first display in Lemma~\ref{l.lip_unif_N} ensures that $\bar F_N$ is Lipschitz uniformly in $N$. The second display indicates that $(\partial_t, \nabla)\bar F_N$ is  ``nonnegative'' in the sense of the following partial orders.
On $\S^K$ and on $\R\times \S^K$, we declare
\begin{align}
    h_1 \leq h_2 \quad&\Longleftrightarrow\quad h_2 - h_1 \in \S^K_+;\label{e.partial_oder_1}\\
    (t_1, h_1)\leq (t_2,h_2) \quad&\Longleftrightarrow\quad (t_2,h_2)-(t_1, h_1)\in [0,\infty)\times \S^K_+.\label{e.partial_oder_2}
\end{align}
{As a consequence of Lemma~\ref{l.lip_unif_N} and the mean value theorem, we have that
\begin{align}\label{e.bar_F_nondec}
    \bar F_N\text{ is nondecreasing,}\quad\forall N
\end{align}
in the sense given in \eqref{e.partial_oder_2}.}

The next result shows that $(\partial_t, \nabla)\bar F_N$ is ``nondecreasing''.

\begin{lemma}\label{l.F_N.order}
For each $N\in\N$, for every $(t_1,h_1)\leq (t_2,h_2)$, it holds that
\begin{align*}
    (\partial_t,\nabla )\bar F_N(t_1,h_1)\leq (\partial_t,\nabla )\bar F_N(t_2,h_2).
\end{align*}
\end{lemma}

\begin{proof}

For $k =1,2$, we set
\begin{align*}
     Y_k  :=\left(\sqrt{\frac{2t_k}{N^{\mathsf{p}-1}}}X^{\otimes \mathsf{p}}A + W_k  \,,\  X\sqrt{2h_k}+Z_k\right)
\end{align*}
where $W_k$ and $Z_k$ consist of i.i.d.\ standard Gaussian random variables. For $k=1,2$, denoting $\la\,\cdot\,\ra$ evaluated at $(t_k,h_k)$ by $\la \,\cdot\,\ra_{k}$, we have
\begin{align}\label{e.<g>_k}
    \la g(x)\ra_{k} = \E[g(X)\,|\,Y_k]
\end{align}
for any measurable function $g$ satisfying $\E|g(X)|<\infty$. 
For any matrix $y$, we write $\,\mathbf{c}(y) :=  y^\intercal y$. Note that $\,\mathbf{c}(X^{\otimes \mathsf{p}}A) \in \R^{L\times L}$ and $\,\mathbf{c}(X) \in \R^{K\times K}$. 
Then, we have
\begin{align*}
    (\partial_t,\nabla)\bar F_N(t_k,h_k) = \E\left(\frac{1}{N^\mathsf{p}} \tr  \,\mathbf{c}\left(\la X^{\otimes \mathsf{p}}A \ra_k\right),\ \frac{1}{N}  \mathbf{c}\left(\la X\ra_k\right)\right).
\end{align*}
Hence, it suffices to show that, for any measurable $g$ satisfying $\E|g(X)|<\infty$,
\begin{align}\label{eq:cd_Y_1<cd_Y_2}
    \E \,\mathbf{c}\left(\la g(X)\ra_1\right) \leq \E \,\mathbf{c}\left(\la g(X)\ra_2\right).
\end{align}
Indeed, in view of the previous display, the desired result follows from taking $g$ to be $g(q) = q^{\otimes \mathsf{p}}A$ and then the identity map.

To compare the two sides in \eqref{eq:cd_Y_1<cd_Y_2}, we introduce
\begin{align*}
     Y' := \left(\sqrt{\frac{2t_2-2t_1}{N^{\mathsf{p}-1}}}X^{\otimes \mathsf{p}}A + W'  ,\  X\sqrt{2h_2-2h_1}+Z'\right),
\end{align*}
where $W'$ and $Z'$ have i.i.d.\ standard Gaussian entries, {independent of randomness previously introduced}. We claim that
\begin{align}\label{eq:id_cd_exp}
     \E[g(X)\,|\,Y_2] \stackrel{\mathrm{d}}{=} \E[g(X)\,|\,Y_1, Y'],
\end{align}
where the equality holds in the sense of probability distributions.
Temporarily assuming this, and using that $\E[g( X)\,|\, Y_1] = \E\big[ \E [g(X)\,|\, Y_1,Y']\,\big|\,Y_1\big]$, we can verify, analogously to a bias-variance decomposition, that
\begin{align*}
    \E \,\mathbf{c}\big(\E[g(X)\,|\, Y_1,Y']\big) = \E \,\mathbf{c}\big(\E[g(X)\,|\, Y_1,Y']-\E[g(X)\,|\, Y_1]\big)+\E \,\mathbf{c}\big(\E[g(X)\,|\, Y_1]\big).
\end{align*}
Since the first term on the right is a positive semi-definite matrix, we get that
\begin{align*}
    \E \,\mathbf{c}\big(\E[g(X)\,|\, Y_1,Y']\big) \geq\E \,\mathbf{c}\big(\E[g(X)\,|\, Y_1]\big).
\end{align*}
In view of \eqref{e.<g>_k} and \eqref{eq:id_cd_exp}, this yields \eqref{eq:cd_Y_1<cd_Y_2} and thus the desired result.

It remains to prove \eqref{eq:id_cd_exp}. The quantities on both sides can be written as integrations of $f$ with respect to Gibbs measures with a common reference measure $P^X_N$ (the law of~$X$). Hence, it suffices to compare the Hamiltonians. The Hamiltonian for the left-hand side can be computed to be
\begin{gather*}
    \frac{2t_2}{N^{\mathsf{p}-1}}(x^{\otimes\mathsf{p}}A)\cdot (X^{\otimes  \mathsf{p}}A)+\frac{1}{\sqrt{N^{\mathsf{p}-1}}}(x^{\otimes \mathsf{p}}A)\cdot \sqrt{2t_2} W_2 - \frac{t_2}{N^{\mathsf{p}-1}}|x^{\otimes \mathsf{p}}A|^2\\+
    {2}h_2\cdot (x^\intercal X)+  (Z_2\sqrt{2h_2})\cdot x- {h_2\cdot (x^\intercal x)},
\end{gather*}
while the Hamiltonian for the right-hand side is
\begin{gather*}
    \frac{2t_2}{N^{\mathsf{p}-1}}(x^{\otimes\mathsf{p}}A)\cdot (X^{\otimes  \mathsf{p}}A)+\frac{1}{\sqrt{N^{\mathsf{p}-1}}}(x^{\otimes\mathsf{p}}A)\cdot \Big(\sqrt{2t_1}W_1+\sqrt{2t_2-2t_1} W'\Big) - \frac{t_2}{N^{\mathsf{p}-1}}|x^{\otimes \mathsf{p}}A|^2\\ + {2}h_2\cdot (x^\intercal X)+  \Big(Z_1\sqrt{2h_1} + Z'\sqrt{2h_2-2h_1}\Big)\cdot x- {h_2\cdot (x^\intercal x)}.
\end{gather*}
Since $W_1,W_2,W',Z_1,Z_2,Z'$ all consist of i.i.d.\ standard Gaussian entries, we can conclude that the two Hamiltonians have the same distribution, which implies \eqref{eq:id_cd_exp}.
\end{proof}

\subsection{Convexity}

In this subsection, we show the following.

\begin{lemma}\label{l.F_N.conv}
For each $N\in\N$, the function $\bar F_N:[0,\infty)\times \S^K_+\to \R$ is convex.
\end{lemma}

\begin{proof}

We want to show that for every $(s,a)\in \R\times \S^K$ and every $(t,h)\in [0,\infty)\times \S^K_+$,
\begin{align*}
    \big(s\partial_t + a\cdot\nabla\big)^2\bar F_N(t,h)\geq 0.
\end{align*}
For brevity, we set $y =\sqrt{\frac{2}{N^{\mathsf{p}-1}}} x^{\otimes \mathsf{p}}A$ and similarly for replicas of $x$. We can compute that
\begin{gather*}
    s^2 \partial_t^2 \bar F_N(t,h) = \frac{2s^2}{N} \E \la ( y \cdot y')( y\cdot y ' - 2 y\cdot y'' + y''\cdot y''')\ra,\\
    s \partial_t \big(a\cdot\nabla\bar F_N(t,h)\big) = \frac{2s}{N} \E \la (a\cdot x^\intercal x' )( y\cdot y ' - 2 y\cdot y'' + y''\cdot y''')\ra,\\
    (a\cdot\nabla)^2\bar F_N(t,h) = \frac{2}{N} \E \la (a\cdot x^\intercal x' )(a\cdot x^\intercal x'  - 2a\cdot x^\intercal x''  +a\cdot x''^\intercal x''' )\ra.
\end{gather*}
Again, this computation uses the Nishimori identity and the Gaussian integration {by} parts. Details for deriving the third identity above can be seen in the derivation of \cite[(3.27)]{mourrat2019hamilton}. The two other identities can be computed by following the same procedure.
Let $I$ be the identity matrix of the same size as $y^\intercal y'$. Setting $b = \mathrm{diag}(a,sI)$, $z=\mathrm{diag}(x,y)$ and similarly for replicas, we have $b\cdot z^\intercal z' = sy\cdot y'+a\cdot x^\intercal x'$ (where the matrix product is carried out prior to the dot product). In this notation, adding the above identities together and using the symmetry between replicas, we have
\begin{align*}
    &\big(s\partial_t + a\cdot\nabla\big)^2\bar F_N(t,h)= \frac{2}{N}\E \la (b\cdot z^\intercal z')^2 -2(b\cdot z^\intercal z')(b\cdot z^\intercal z'') +(b\cdot z^\intercal z')(b\cdot z''^\intercal z''')\ra\\
    &= \frac{2}{N}\E\la (b\otimes b)\cdot\Big(z^\intercal z'\otimes z^\intercal z' - 2z^\intercal \langle z'\rangle \otimes  z^\intercal \langle z'\rangle + \langle z\rangle^\intercal \langle z'\rangle \otimes \langle z\rangle^\intercal \langle z'\rangle\Big) \ra.
\end{align*}
Writing $\bar z = z -\langle z \rangle$ and similarly for replicas, we obtain that the above is equal to 
\begin{align*}
    \frac{2}{N}\E\la (b\otimes b)\cdot\Big(z^\intercal \bar z'\otimes z^\intercal \bar z' - \bar z^\intercal \langle z'\rangle \otimes \bar z^\intercal \langle z'\rangle\Big) \ra.
\end{align*}
Since $b$ is symmetric, we can see that
\begin{align*}
    (b\otimes b) \cdot (\bar z^\intercal \langle z'\rangle \otimes \bar z^\intercal \langle z'\rangle) = (b\otimes b) \cdot (\langle z'\rangle^\intercal \bar z \otimes  \langle z'\rangle^\intercal \bar z).
\end{align*}
Using the symmetry between replicas, we conclude from the above three displays that
\begin{equation*}
    \big(s\partial_t + a\cdot\nabla\big)^2\bar F_N(t,h)=\frac{2}{N}\E\la (b\otimes b)\cdot\Big(\bar z^\intercal \bar z'\otimes \bar z^\intercal \bar z' \Big) \ra\geq 0. \qedhere
\end{equation*}
\end{proof}

\section{Some results of convex analysis}\label{s.convex}

As mentioned above, our approach to proving Theorem~\ref{main_thm} relies on the identification of the limit of $\bar F_N$ as the unique viscosity solution to \eqref{e.HJ_tensor}. The uniqueness criterion we will use for this purpose is inspired by that described in Appendix~\ref{a.vis}. Compared with the setting explored there, equation \eqref{e.HJ_tensor} poses additional difficulties that are caused by the fact that the domain $\S^K_+$ of the ``space'' variable has a boundary. This is compounded by the fact that the relevant order on $\S^K_+$ is not total. The main purpose of this section is to demonstrate Proposition~\ref{p.order}, which states that, despite this, the subgradient of a nondecreasing convex function with nondecreasing gradients always has a maximal element (and this maximal element has further good properties). This proposition will be particularly handy in Section~\ref{s.hj}. 

\subsection{Preliminaries}

We start by recalling basic definitions and results from convex analysis.
Since we need results for both functions defined on $\S^K_+$ and functions on $[0,\infty)\times \S^K_+$, we consider a slightly more general setting in this subsection and specialize into these two spaces when needed.

Let $\mathscr H$ be a finite-dimensional Hilbert space. The associated inner product is denoted by a dot product, and the norm by $|\,\cdot\,|$. Since $\mathscr H$ can be isometrically identified with a Euclidean space, the usual notion of differentiability for any function $u:\mathscr H \to \R$ still makes sense. If $u$ is differentiable at $x\in \mathscr H$, we denote by $Du(x)$ its differential at $x$. We also identify $\mathscr H$ with its dual and thus $Du(x)\in \mathscr H$. For the purpose of this work, the space $\mathscr H$ will be taken to be either $\R\times \S^K$ or $\S^K$, and, correspondingly, $D$ will be taken to be either $(\partial_t,\nabla )$ or $\nabla$.

Let $u:\mathscr H\to \R\cup\{\infty\}$ be a convex function. We define its subdifferential at $x\in\mathscr H$ by
\begin{align}\label{e.def.subdiff}
    \partial u(x) := \Big\{y\in\mathscr H:\ u(x')\geq u(x)+ y\cdot(x'-x),\  \forall x'\in\mathscr H\Big\}.
\end{align}
The effective domain of $u$ is
\begin{equation*}
    \mathsf{dom}\, u := \{x \in \mathscr H:u(x)<\infty\}.
\end{equation*}
The function $u$ is called \textit{proper} if $\mathsf{dom}\, u\neq \emptyset$. The outer normal cone to a subset $\mathscr S \subset \mathscr H$ at $x\in\mathscr H$ is given by
\begin{align}\label{e.def_outer_normal}
    \mathbf{n}_{\mathscr S}(x) := \big\{y\in\mathscr H:\  y\cdot(x'-x)\leq 0,\ \forall x'\in \mathscr S\big\}.
\end{align}

The following result characterizes the subdifferential as the sum of the outer normal cone and the set of accumulation points of differentials at nearby differentiable points; we refer to  \cite[Theorem 25.6]{rockafellar1970convex} for a proof.

\begin{lemma}\label{l.char_subd}
Let $u:\mathscr H\to \R\cup\{\infty\}$ be a proper lower semi-continuous convex function such that $\mathsf{dom}\, u$ has nonempty interior. Then, for every $x \in \mathscr H$,
\begin{align*}
    \partial u(x) = \mathsf{cl}(\mathsf{conv}\, S) + \mathbf{n}_{\mathsf{dom}\, u}(x).
\end{align*}
where $S$ is the set of all limits of sequences of the form $\big(Du(x_i)\big)_{i\in\N}$ such that $u$ is differentiable at $x_i$ and $\lim_{i\to\infty}x_i = x$.
\end{lemma}
{Note that when $x$ is in the interior of $\mathsf{dom}\,u$, we have $\mathbf{n}_{\mathsf{dom}\, u}(x)=\{0\}$.}



We also record two classical results which, while not relevant to the proof of Proposition~\ref{p.order}, will be useful later on. The first one characterizes the subdifferential of the sum of two convex functions, assuming that one of them is differentiable for simplicity. The second one states a correspondence between elements of the subdifferential at a point and smooth functions that ``touch the convex function from below''.    

\begin{lemma}\label{l.convex+diff}
Let $u:\mathscr H\to \R\cup\{\infty\}$ be a proper lower semi-continuous convex function such that $\mathsf{dom}\, u$ has nonempty interior. Let $v:\mathscr H\to \R$ be convex and differentiable everywhere. Set $u'=u+v$. Then, $\mathsf{dom}\, u=\mathsf{dom}\, u'$ and, for every $x\in \mathsf{dom}\, u$, it holds that
\begin{align*}
    \partial u'(x) = \partial u(x) +\big\{Dv(x)\big\}.
\end{align*}
\begin{proof}
The first claim is obvious due to the finiteness of $v$. To see the second claim, we start by noting that due to $\mathsf{dom}\, u=\mathsf{dom}\, u'$, the outer normal cone to $\mathsf{dom}\, u$ is the same as the outer normal cone to $\mathsf{dom}\, u'$ at every point. The differentiability of $v$ implies that $u'$ is differentiable at some point $x'$ if and only if $u$ is also differentiable at $x'$. Hence, the second claim follows from Lemma~\ref{l.char_subd}.
\end{proof}
\end{lemma}

\begin{lemma}\label{l.subsol_subd}
Let $u:\mathscr H\to \R\cup\{\infty\}$ be convex. Then, $p \in \partial u(x)$ for some $x$ if and only if there exists a smooth function $\phi:\mathscr H\to \R$ such that $u-\phi$ achieves its minimum at $x$ and $D\phi(x) = p$.
\end{lemma}
\begin{proof}
Assuming $p\in\partial u(x)$, we can deduce from the definition of subdifferential that $u-\phi$ achieves {its} minimum at $x$ for $\phi:y\mapsto p\cdot y$. Now, let us assume the converse. The convexity of $u$ implies that
\begin{align*}
    u(x') -u(x) \geq \frac{1}{\lambda}\Big(u\big(x+\lambda(x'-x)\big)-u(x)\Big),\quad\forall x',\ \forall \lambda \in(0,1].
\end{align*}
Using the minimality of $u-\phi$ at $x$ and the differentiability of $\phi$ at $x$, we can obtain $D\phi(x)\in\partial u(x)$ by sending $\lambda \to 0$.
\end{proof}

To apply these results to the study of solutions to \eqref{e.HJ_tensor}, we make the following remark.

\begin{remark}\label{r.extension}
Any convex function $f:[0,\infty)\times\S^K_+\to\R$ can be extended in a standard way to a convex function $\bar f:\R\times \S^K\to \R\cup\{\infty\}$ by setting $\bar f = f$ on $[0,\infty)\times\S^K_+$ and $\bar f = \infty$ elsewhere. Note that $\bar f$ is proper and its effective domain is $[0,\infty)\times\S^K_+$ which has nonempty interior. If $f$ is continuous, then $\bar f$ is lower semi-continuous. In the following, we do not distinguish between $f$ and its standard extension. Then, the notions and results discussed above can be applied to $f$ by setting $\mathscr H= \R\times \S^K$ and $D = (\partial_t,\nabla)$. Similar treatments can be taken for any convex function $\psi: \S^K_+\to \R$.
\end{remark}

Finally, since we will work with functions defined on $\S^K_+$ and $[0,\infty)\times \S^K_+$, we record these two simple lemmas.

\begin{lemma}\label{l.psd}
For every $a \in \S^K$, we have $a\in \S^K_+$ if and only if $a\cdot b \geq 0$ for all $b\in\S^K_+$.
\end{lemma}

\begin{lemma}\label{l.outer_normal}
For every $t\geq 0$ and $x \in \S^K_+$, we have $\mathbf{n}_{\S^K_+}(x) \subset - \S^K_+$ and $\mathbf{n}_{[0,\infty)\times\S^K_+}(t,x) \subset -([0,\infty)\times \S^K_+)$.
\end{lemma}

The first lemma is an application of the diagonalizability of real symmetric matrices (see e.g.\ \cite[Lemma 2.2]{mourrat2019hamilton}), and the second lemma is a consequence of {the first lemma} and the definition of outer normal cones in \eqref{e.def_outer_normal}.

\medskip

\subsection{Nondecreasing gradients}
The key result of this subsection is Proposition~\ref{p.order}. To state it, it is convenient to introduce the following definitions. Recall the partial orders defined in \eqref{e.partial_oder_1} and \eqref{e.partial_oder_2}.
\begin{definition}[Nondecreasingness]\label{d.nondecreasing}
A real-valued function $g$ defined on $\S^K_+$ or $[0,\infty)\times \S^K_+$ is said to be nondecreasing if $g(y_1)\leq g(y_2)$ whenever $y_1\leq y_2$.
\end{definition}
\begin{definition}[Nondecreasing gradients]
A Lipschitz function $f:[0,\infty)\times \S^K_+\to \R$ is said to have nondecreasing gradients if, for every $(t_1,x_1)$ and $(t_2,x_2)$ that are {points of differentiability} of $f$ and satisfy $(t_1,x_1)\leq (t_2,x_2)$, it holds that
\begin{align}\label{e.order}
    (\partial_t,\nabla)f(t_1,x_1)\leq (\partial_t,\nabla)f(t_2,x_2). 
\end{align}
\end{definition}

Recall that, by Rademacher's theorem, a Lipschitz function is differentiable almost everywhere. Here is the main result of this section.

\begin{proposition}\label{p.order}
Suppose that $f:[0,\infty)\times \S^K_+\to \R$ is nondecreasing, Lipschitz, convex, and has nondecreasing gradients. Then,
for every $(t,x)\in [0,\infty)\times \S^K_+ $, there exists $(b,q)\in\partial f(t,x)\cap [0,\infty)\times \S^K_+$ such that $|(b,q)|\leq \|f\|_{\mathrm{Lip}}$ and
\begin{align}\label{e.l.order}
    \text{for every }  (a,p)\in \partial f(t,x), \quad (a,p)\leq (b,q).
\end{align}
In addition, if $f$ satisfies \eqref{e.HJ_tensor} on a dense set, then $(b,q)$ can be chosen to satisfy $b-\H(q)= 0$.
\end{proposition}

\begin{remark}
\label{r.ae.diff}
In the statement of Proposition~\ref{p.order}, the precise interpretation of the phrase that $f$ satisfies \eqref{e.HJ_tensor} on a dense set is that the set
\begin{equation*}  
\Ll\{ (t,x) \in (0,\infty)\times \S^K_{++} \ : \ \mbox{$f$ is differentiable at $(t,x)$ and }  \Ll( \dr_t f - \msf H(\nabla f) \Rr) (t,x) = 0 \Rr\} 
\end{equation*}
is dense in $[0,\infty) \times \S^K_+$. We point out that one could equivalently replace this condition by the condition that $f$ satisfies \eqref{e.HJ_tensor} at every point of differentiability in $(0,\infty) \times \S^K_{++}$. Indeed, one direction of this equivalence is immediate, since every Lipschitz function is differentiable almost everywhere. Conversely, if $(t,x) \in (0,\infty) \times \S^K_{++}$ is a point of differentiability of $f$, one can find a sequence of points $(t_n,x_n)$ that converge to $(t,x)$ and such that \eqref{e.HJ_tensor} is satisfied at $(t_n,x_n)$. Then every subsequential limit of $(\dr_t f, \nabla f)(t_n,x_n)$, say $(a,p) \in \R \times \S^K$, satisfies $a - \msf H(p) = 0$, and one can check that $(a,p) \in \dr f(t,x)$. But since $f$ is differentiable at $(t,x)$ and $(t,x)$ is in the interior (implying that the outer normal cone is \{0\}), the subdifferential $\dr f(t,x)$ is the singleton $\{(\dr_t , \nabla )f(t,x)\}$. 
\end{remark}

\begin{proof}[Proof of Proposition~\ref{p.order}]

Let $(t,x)\in [0,\infty)\times \S^K_+$. We start by fixing some $(s_0,y_0)\in (0,\infty)\times \S^K_{++}$ such that $|(s_0,y_0)|=1$. Note that
\begin{align*}
    (t,x) +\lambda (s_0,y_0)\in [0,\infty)\times\S^K_+,\quad\forall \lambda \geq 0.
\end{align*}
Since $f$ is differentiable a.e.\ on $[0,\infty)\times \S^K_+$, we can find a sequence $(t_{0,j},x_{0,j})_{j\in \N}$ of differentiable points such that
\begin{align}\label{e.t_0,j}
     \big|(t_{0,j},x_{0,j})-\big((t,x)+j^{-1}(s_0,y_0)\big)\big|\leq j^{-2},\quad\forall j\in\N.
\end{align}
If, in addition, $f$ satisfies \eqref{e.HJ_tensor} on a dense set, then clearly we can choose $(t_{0,j},x_{0,j})_{j\in\N}$ from that set. Since $f$ is Lipschitz, by passing to a subsequence, we may assume that $\lim_{j\to\infty}(\partial_t,\nabla)f(t_{0,j},x_{0,j})$ exists. Denote this limit by $(b,q)$. By Lemma~\ref{l.char_subd}, we know that $(b,q)\in\partial f(t,x)$. It is clear that $|(b,q)|\leq \|f\|_{\mathrm{Lip}}$. Since $f$ is nondecreasing, we also have that $(b,q)\in[0,\infty)\times \S^K_+$. Continuity of $\H$ implies that $b-\H(q)=0$ if $f$ satisfies \eqref{e.HJ_tensor} on a dense set. It remains to show \eqref{e.l.order}.

We apply Lemma~\ref{l.char_subd} to the standard extension of $f$ {(see Remark~\ref{r.extension})}. Note that $\mathsf{dom}\, f= [0,\infty)\times \S^K_+$. Let $S$ be the corresponding set at $(t,x)$ in this lemma.
Then, due to this and Lemma~\ref{l.outer_normal}, for each $(a,p) \in \partial f(t,x)$, there is $(a',p')\in \mathsf{cl}(\mathsf{conv}\, S)$ such that $(a,p)\leq (a',p')$. Therefore, it suffices to prove \eqref{e.l.order} for $(a,p)\in \mathsf{cl}(\mathsf{conv}\, S)$. In fact, since the condition on $(a,p)$ in~\eqref{e.l.order} is stable under convex combinations and passage to the limit, it suffices to show~\eqref{e.l.order} for every $(a,p) \in S$. 

Let $(a,p)\in S$. By definition of $S$, there exists a sequence $((t_i,x_i))_{i\in\N}$ converging to~$(t,x)$ such that
\begin{gather}
    \lim_{i\to\infty} (\partial_t,\nabla)f(t_i,x_i) = (a,p).\label{e.a,p}
\end{gather}
Due to our choice of $(s_0,y_0)$, we can see that for sufficiently large $j$ there is $i(j)\in\N $ such that
\begin{align}\label{e.t,x<}
    (t_i,x_i)\leq (t_{0,j},x_{0,j}),\quad\forall i\geq i(j).
\end{align}
Indeed, since $(s_0,y_0)$ is strictly positive, there is $C>0$ such that
\begin{align*}
    C^{-1}|(a',p')|\leq (a',p')\cdot (s_0,y_0)\leq C|(a',p')|,\quad\forall (a',p')\in[0,\infty)\times\S^K_+.
\end{align*}
By this and \eqref{e.t_0,j}, we have that, for every $\mathbf a\in[0,\infty)\times\S^K_+$,
\begin{align*}
    \mathbf a\cdot\bigg((t_{0,j},x_{0,j})-(t,x) - \frac{1}{2j}(s_0,y_0)\bigg)\geq \frac{1}{2j}\mathbf a\cdot(s_0,y_0) - j^{-2}|\mathbf a|\geq |\mathbf a|\bigg(\frac{1}{2Cj}-\frac{1}{j^2}\bigg).
\end{align*}
The right-hand side is nonnegative for sufficiently large $j$. Lemma~\ref{l.psd} thus implies that
\begin{align*}
    (t_{0,j},x_{0,j})-(t,x)\geq \frac{1}{2j}(s_0,y_0).
\end{align*}
On the other hand, similar arguments yield that, for sufficiently large $i$ (in terms of $j$), 
\begin{align*}
    (t_i,x_i)-(t,x) \leq \frac{1}{2j}(s_0,y_0).
\end{align*}
The two previous displays justify \eqref{e.t,x<}.
Using \eqref{e.a,p}, \eqref{e.t,x<} and the property \eqref{e.order}, by first sending $i\to\infty$ and then $j\to\infty$, we obtain that
\begin{equation*}
    (a,p)\leq \lim_{j\to\infty}(\partial_t,\nabla)f(t_{0,j},x_{0,j})=(b,q),
\end{equation*}
as desired.
\end{proof}

\section{Viscosity solutions}\label{s.hj}




In this section, we study the Hamilton-Jacobi equation \eqref{e.HJ_tensor}. First, we give the precise definition of viscosity solutions. Then, we recall the uniqueness and existence of viscosity solutions ensured by the comparison principle and the fact that the Hopf formula gives a viscosity solution. We next turn to the main goal of this section, which is to prove Proposition~\ref{p.weak-tensor}. This proposition provides us with a convenient sufficient condition for checking whether a function is the unique viscosity solution. This is instrumental in our proof of the convergence of the free energy in Section~\ref{s.proof}.


Recall that the notion of nondecreasing functions was introduced in Definition~\ref{d.nondecreasing}.
\begin{definition}[Viscosity solutions]\label{def:vs}
\leavevmode
\begin{enumerate}
    \item A nondecreasing Lipschitz function $f:[0,\infty)\times \S^K_+\to \R$ is a viscosity subsolution to \eqref{e.HJ_tensor} if for every $(t,h) \in (0,\infty)\times \S^K_+$ and every smooth $\phi:(0,\infty)\times \S^K_+ \to \R$ such that $f-\phi$ has a local maximum at $(t,h)$, we have
\begin{align*}
    \begin{cases}
    \big(\partial_t \phi - \H(\nabla\phi)\big)(t,h)\leq 0, \quad &\text{if }h\in \S^K_{++},\\
     \nabla\phi(t,h) \in\S^K_+,  \quad &\text{if }h\in \S^K_+\setminus\S^K_{++}.
    \end{cases}
\end{align*}

\item \label{item:vs_2} A nondecreasing Lipschitz function $f:[0,\infty)\times \S^K_+\to \R$ is a viscosity supersolution to \eqref{e.HJ_tensor} if for every $(t,h) \in (0,\infty)\times \S^K_+$ and every smooth $\phi:(0,\infty)\times \S^K_+ \to \R$ such that $f-\phi$ has a local minimum at $(t,h)$, we have
\begin{align*}
    \begin{cases}
    \big(\partial_t \phi - \H(\nabla\phi)\big)(t,h)\geq 0, \quad &\text{if }h\in \S^K_{++},\\
    \partial_t \phi (t,h) - \inf \H(q)\geq 0,\quad ,\quad &\text{if }h\in \S^K_+\setminus\S^K_{++},
    \end{cases}
\end{align*}
where the infimum is taken over all $q\in \big(\nabla \phi(t,h) +\S^K_+\big)\cap \S^K_+  $ and $|q|\leq \|f\|_{\mathrm{Lip}}$. 

\item A nondecreasing Lipschitz function $f:[0,\infty)\times \S^K_+\to \R$ is a viscosity solution to~\eqref{e.HJ_tensor} if $f$ is both a viscosity subsolution and supersolution.

\end{enumerate}
\end{definition}

Remarks~\ref{r.super} and \ref{r.boundary} below aim to provide a somewhat more intuitive understanding of Definition~\ref{def:vs}. Before doing so, we record the following observation.
\begin{lemma}\label{l.DH}
The function $\H:\S^K_+\to \R$ given in \eqref{def.H} is nondecreasing.
\end{lemma}
\begin{proof}
Let $a,b \in \S^K_+$ be such that $a \le b$. Recalling that the tensor product of two positive semidefinite matrices is positive semidefinite, see for instance \cite[Theorem~7.20]{matrix}, one can show by induction on $\mathsf{p}$ that $a^{\otimes \mathsf{p}} \le b^{\otimes \mathsf{p}}$. Since $A A^\intercal \in \S^{K^\mathsf{p}}_+$, we can use Lemma~\ref{l.psd} to obtain that $\H(a) \le \H(b)$, as desired.
\end{proof}
\begin{remark}\label{r.super}
Given a nondecreasing Lipschitz function $f$, define the extension of $\H$ by
\begin{align}\label{e.def.bar_H}
    \bar\H(p) := \inf\Big\{\H(q):q\geq p,\, q\in\S^K_+,\, |q|\leq \|f\|_{\mathrm{Lip}}\Big\},\quad\forall p\in\S^K.
\end{align}
As usual, the infimum over an empty set is understood to be $\infty$. Note that $\bar \H:\S^K\to\R\cup\{\infty\}$ is lower semi-continuous and agrees with $\H$ on $\S^K_+$ due to Lemma~\ref{l.DH}.
Then, Definition~\ref{def:vs}~\eqref{item:vs_2} can be reformulated as follows: $f$ is {a} viscosity supersolution if for every $(t,h) \in (0,\infty)\times \S^K_+$ and every smooth $\phi:(0,\infty)\times \S^K_+$ such that $f-\phi$ has a local minimum at $(t,h)$, we have
\begin{align*}
    \big(\partial_t\phi - \bar\H(\nabla\phi)\big)(t,h)\geq 0.
\end{align*}
Note that, in this formulation, we do not need to distinguish between $h\in \S^K_{++}$ and $h\in \S^K_+\setminus\S^K_{++}$.
\end{remark}
\begin{remark}  
\label{r.boundary}
Further simplifications of boundary conditions can be made. After the submission of this paper, \cite{chen2022hamilton} considers solutions defined to satisfy the equation in the viscosity sense everywhere including the boundary without any additional boundary condition imposed. Under this definition, the comparison principle and the existence of solutions still hold. Moreover, the solution admits a representation by the Hopf--Lax formula given the convexity of the nonlinearity, or the Hopf formula given the convexity of the initial condition. 
All properties needed in this work are still satisfied. One can work with this definition, and the main results in this work are still valid.

Let us briefly describe the simplification. Due to Lipschitzness of $\bar F_N$ uniformly in $N$ (Lemma~\ref{l.lip_unif_N}), we can work with a regularized nonlinearity $\H^{\mathsf{reg}}:\S^D_+\to\R$ which coincides with $\H$ on a ball intersected with $\S^D_+$ with sufficiently large radius. In a similar way as in \cite[Lemma~4.2]{chen2022hamilton}, $\H^{\mathsf{reg}}$ can be constructed to be Lipschitz and nondecreasing. Then, we extend $\H^{\mathsf{reg}}$ to
\begin{align*}
    \H^{\mathsf{ext}}(p):=\inf\left\{\H^{\mathsf{reg}}(q):q\geq p,\, q\in\S^K_+\right\},\quad\forall p\in\S^K.
\end{align*}
One can check, similarly as in \cite[Lemma~4.4]{chen2022hamilton}, that $\H^{\mathsf{ext}}$ is Lipschitz and and nondecreasing. Then, the conditions for viscosity subsolutions and supersolutions can be replaced by
\begin{gather*}
    \left(\partial_t \phi - \H^{\mathsf{ext}}(\nabla\phi)\right)(t,h)\leq 0,
    \\
    \left(\partial_t \phi - \H^{\mathsf{ext}}(\nabla\phi)\right)(t,h)\geq 0,
\end{gather*}
respectively, without the need to distinguish between $h\in \S^K_+\setminus\S^K_{++}$ and $h\in \S^K_{++}$. The key property needed for this simplification in \cite{chen2022hamilton} is the monotonicity of the nonlinearity.

\end{remark}

We turn to the well-posedness of equation \eqref{e.HJ_tensor}. We first state a comparison principle, which ensures in particular that there is at most one viscosity solution with a given initial condition. 
\begin{proposition}[Comparison principle]\label{p.comp}
If $u$ is a {subsolution} and $v$ is a {supersolution} to \eqref{e.HJ_tensor}, then
\begin{equation*}
    \sup_{[0,\infty)\times\S^K_+}(u-v)=\sup_{\{0\}\times\S^K_+}(u-v).
\end{equation*}
\end{proposition}
For suitable initial conditions, the viscosity solution admits the following variational representation.
\begin{proposition}[Hopf formula]\label{p.hopf}
Let $\psi:\S^K_+\rightarrow \R$ be convex, Lipschitz and nondecreasing, and let $f$ be given by
\begin{equation*}
    f(t,h) := \sup_{h''\in\S^K_+}\inf_{h'\in\S^K_+}\big\{h''\cdot (h-h')+\psi(h')+t\H(h'')\big\},\quad \forall (t,h)\in[0,\infty)\times \S^K_+.
\end{equation*}
Then, the function $f$ is a viscosity solution to \eqref{e.HJ_tensor} with initial condition $f(0,\cdot )=\psi$.
\end{proposition}

For the proofs of these two propositions, we refer to \cite[Section~6]{HBJ}. 

In the {remainder} of this section, for convenience, we will use $x,y$ as spatial variables in place of $h$, which should not be confused with the notation for random variables under the Gibbs measure $\la\,\cdot\,\ra$ in Section~\ref{s.prop.F_N}.

\subsection{Identification criterion}\label{s.suff_cond}

The following result gives a convenient criterion for a function to be a viscosity solution.

\begin{proposition}
\label{p.weak-tensor}
Let $f : [0,\infty) \times \S^K_+ \to \R$ be nondecreasing, Lipschitz, convex, and have nondecreasing gradients. Suppose that $\psi = f(0,\cdot)$ is $C^1$ and that $f$ satisfies \eqref{e.HJ_tensor} on a dense subset. Then, $f$ is a viscosity solution to \eqref{e.HJ_tensor} with initial condition $\psi$. 
\end{proposition}

For the reader's convenience, the idea for the proof of this proposition is also presented in the simpler setting of Hamilton-Jacobi equations on $[0,\infty)\times \R^d$ in Appendix~\ref{a.vis}. Two essential ingredients for this argument are the $C^1$ assumption of the initial condition and the convexity of $f$. At least in the simpler context explored in Appendix~\ref{a.vis}, both assumptions are necessary; see in particular Remark~\ref{r.diff.assumption} there.

Compared with the Euclidean setting discussed in Appendix~\ref{a.vis}, the existence of the boundary of $\S^K_+$ complicates the arguments. Indeed, in view of Lemma~\ref{l.char_subd}, on the boundary, the subdifferential contains an additional component from the outer normal cone. Therefore, if $p\in\partial \psi(y)$ for a boundary point $y$, we cannot identify $p$ with $\nabla \psi(y)$. The identity $p=\nabla \psi (y)$ is important in Step~2 of the proof of Proposition~\ref{p.weak}. It turns out that for Proposition~\ref{p.weak-tensor}, a work-around is available by exploiting the assumption that the function $f$ has nondecreasing gradients.

As preparation for this, we use Proposition~\ref{p.order} to prove the following lemma. This lemma can be interpreted as stating that we can always ``lift'' a subdifferential $p\in\partial\psi(y)$ to a subdifferential $(b,p)\in\partial f(0,y)$ which is dominated by some $(b,p')\in\partial f(0,y)$ satisfying the Hamilton-Jacobi equation. This lemma is needed due to the presence of boundary. Indeed, on $[0,\infty)\times \R^d$, the existence of such a ``lift'' is automatic, which can be seen in Step~2 of the proof of Proposition~\ref{p.weak}.
 

\begin{lemma}\label{l.lift}
Under the assumptions in Proposition~\ref{p.weak-tensor}, for every $y\in\S^K_+$ and every $p\in \partial \psi(y)$, there is $(b,p')\in [0,\infty)\times\S^K_+$ such that $(b,p)\in \partial f(0,y)$, $p'\geq p$, $|(b,p')|\leq \|f\|_{\mathrm{Lip}}$ and $b-\H(p')= 0$.
\end{lemma}

\begin{proof}
Since $\psi:\S^K_+\to \R$ is $C^1$, by Lemma~\ref{l.char_subd} and setting $p' = \nabla \psi(y)$, we have
\begin{align*} 
    \partial \psi(y) = \{p'\} + \mathbf{n}_{\S^K_+}(y).
\end{align*}
This implies that
\begin{align}\label{e.p=p+n}
    p = p' + n
\end{align}
for some
\begin{align}\label{e.n}
    n \in \mathbf{n}_{\S^K_+}(y).
\end{align}
Due to Lemma~\ref{l.outer_normal}, we have $-n \in \S^K_+$, that is,
\begin{align}\label{e.p<p'}
    p\leq p'.
\end{align}
The same argument also yields that,
\begin{align}\label{e.p'>q}
\text{for every }  q' \in \partial \psi(y), \quad    q' \leq p'.
\end{align}
Since $f$ is nondecreasing, we have that, for all $(t',x')\in[0,\infty)\times \S^K_+$,
\begin{align*}
    f(t',x') - f(0,y)\geq f(0,x')-f(0,y) = \psi(x') - \psi(y),
\end{align*}
which {due to the convexity of $\psi$} implies that $(0,p')\in \partial  f(0,y)$. Let
\begin{align}\label{e.(b,q)}
    (b,q) \in \partial f(0,y)
\end{align}
be as described in Proposition~\ref{p.order}, for $f$ at the point $(0,y)$. Then, the following properties hold
\begin{gather}
    (0,p') \leq (b, q),\label{e.p'<q}\\
    |(b,q)|\leq \|f\|_{\mathrm{Lip}},\quad b - \H(q)=0. \label{e.b-H(q)}
\end{gather}
Since $f(0,\cdot) = \psi$, we must have $q \in \partial \psi(y)$. Combining \eqref{e.p'>q} and~\eqref{e.p'<q}, we see that
\begin{align}\label{e.p'=q}
    p' = q.
\end{align}
We are now ready to conclude. 
{By \eqref{e.n} and the definition of outer normal in \eqref{e.def_outer_normal}}, we can verify that
\begin{align*}
    (0,n) \in \mathbf{n}_{[0,\infty)\times \S^K_+}(0,y).
\end{align*}
This along with Lemma~\ref{l.char_subd}, \eqref{e.(b,q)} and \eqref{e.p'=q} implies
\begin{align*}
    (b,p' + n) \in \partial f(0,y).
\end{align*}
The lemma then follows from this display, \eqref{e.p=p+n}, \eqref{e.p<p'}, \eqref{e.b-H(q)} and \eqref{e.p'=q}.
\end{proof}

We are now ready to prove our criterion for the identification of solutions.

\begin{proof}[Proof of Proposition~\ref{p.weak-tensor}]
We check that $f$ must be a subsolution to \eqref{e.HJ_tensor}. Let $\phi \in C^\infty((0,\infty) \times \S^K_+)$, and $(t,x) \in (0,\infty) \times \S^K_+$ be such that $f-\phi$ has a  local maximum at $(t,x)$. If $x\in \S^K_+\setminus \S^K_{++}$, since, for each $a \in \S^K_+$ and sufficiently small $\eps>0$,
\begin{align*}
    0\leq f(t,x+\eps a)-f(t,x) \leq \phi(t,x+\eps a)-\phi(t,x),
\end{align*}
we must have $a\cdot\nabla\phi(t,x)\geq0$ for all $a\in\S^K_+$. By Lemma~\ref{l.psd}, this implies that $\nabla\phi(t,x)\in\S^K_+$. If $x\in \S^K_{++}$, then we have, 
\begin{equation*}  
f(t',x') - f(t,x) \le (t'-t) \dr_t \phi(t,x) + (x'-x)\cdot  \nabla \phi(t,x) + o(|t'-t| + |x'-x|).
\end{equation*}
This implies that the subdifferential $\dr f(t,x)$ is the singleton $\{(\dr_t\phi,\nabla \phi )(t,x)\}$, and thus that~$f$ is differentiable at $(t,x)$, with {$(\dr_t f,\nabla f )(t,x) = (\dr_t \phi,\nabla \phi )(t,x)$}. Using also Remark~\ref{r.ae.diff}, we deduce that 
\begin{equation*}  
\big(\dr_t \phi - \msf H(\nabla \phi)\big)(t,x) = \big(\dr_t f - \msf H(\nabla f)\big)(t,x) = 0,
\end{equation*}
as desired.

Now we want to show that $f$ is a supersolution to \eqref{e.HJ_tensor}.
Fix any $(t,x)$, and any
\begin{align}\label{e.(a,p)}
    (a,p)\in\partial f(t,x).
\end{align}
Recall Remark~\ref{r.super} and the extension $\bar\H$ defined there.
{Taking $(t,x)$ and $\phi$ as in Definition~\ref{def:vs}~\eqref{item:vs_2}, we can use Lemma~\ref{l.subsol_subd} to see that $(\partial_t \phi(t,x),\nabla\phi(t,x))\in\partial f(t,x)$. Therefore, it suffices to show that }
\begin{align}\label{e.a-H(p)}
    a - \bar\H(p)\geq 0.
\end{align}
We proceed in four steps.

\textit{Step 1.} We claim that, for every $\eps>0$, the following infimum
\begin{equation}\label{e.g_inf}
\inf_{y\in\S^K_+}\big(f_\ep(0,y)-y\cdot p\big)
\end{equation}
is achieved, where, for every $(s,y)\in[0,\infty)\times \S^K_+$, we have set
\begin{align*}
    f_\ep(s,y) := f(s,y)+\ep\sqrt{1+|y|^2}.
\end{align*}
{Note that we are working with a slightly different perturbation of $f$ from the one in Step~3 in the proof of Proposition~\ref{p.weak}. The purpose is to ensure that the perturbative term is differentiable everywhere so that Lemma~\ref{l.convex+diff} is applicable.}
One can verify that $y\mapsto \sqrt{1+|y|^2}$ is convex, and thus so is $f_\eps$.
By the definition of subdifferentials, we have
\begin{align*}
    f(0,y)-f(t,x) \geq (a,p)\cdot (-t,y-x),\quad\forall y\in \S^K_+,
\end{align*}
which implies that
\begin{align*}
    f_\eps(0,y)-y\cdot p \geq \eps\sqrt{1+|y|^2}+f(t,x)-(a,p)\cdot(t,x),\quad\forall y\in\S^K_+.
\end{align*}
Hence, the left-hand side of the inequality above is bounded below and tends to infinity as $|y|$ tends to infinity. Therefore, a minimizer exists and we denote it by $y_\ep\in\S^K_+$.

\textit{Step 2.} We show
\begin{equation}
\label{e.eps*y_to_0}
\lim_{\ep \to 0} \ep \sqrt{1+|y_\ep|^2} = 0.
\end{equation}
We first observe that
\begin{equation}  \label{e.limsupin=inf}
\limsup_{\ep\to0} \inf_{y \in \S^K_+}  \Ll(f(0,y)  + \ep \sqrt{1+|y|^2} - y \cdot p\Rr) = \inf_{y \in \S^K_+} \Ll(f(0,y)  - y \cdot p\Rr).
\end{equation}
Indeed, for any $\delta>0$, there is $\bar y\in \S^K_+$ such that
\begin{equation*}
    f(0,\bar y)-\bar y\cdot p\leq\inf(f(0,y)-y\cdot p)+\delta/2,
\end{equation*}
and we can choose $\bar \ep>0$ small enough such that, for every $\ep \in (0,\bar \ep)$,
\begin{equation*}
     f(0,\bar y)+\ep\sqrt{1+|\bar y|^2}-\bar y\cdot p\leq\inf(f(0,y)-y\cdot p)+\delta.
\end{equation*}
This implies that
\begin{equation*}
    \limsup_{\ep\to0} \inf_{y \in \S^K_+}  \Ll(f(0,y)  + \ep \sqrt{1+|y|^2} - y \cdot p\Rr) \leq \inf_{y \in \S^K_+} \Ll(f(0,y)  - y \cdot p\Rr),
\end{equation*}
and the other direction of the inequality in \eqref{e.limsupin=inf} is obvious. 
Since $y_\ep$ achieves the infimum on the left-hand side of \eqref{e.limsupin=inf} and also satisfies
\begin{equation*}  
f(0,y_\ep) - y_\ep \cdot p \ge \inf_{y \in \S^K_+} \Ll(f(0,y)  - y \cdot p\Rr),
\end{equation*}
we conclude that \eqref{e.eps*y_to_0} holds.

\textit{Step 3.} Let $\psi_\eps := f_\eps(0,\cdot)$, so that $\psi_\eps = \psi + \eps\sqrt{1+|\cdot|^2}$. Since $y_\eps$ achieves the infimum in~\eqref{e.g_inf}, we have that $p \in \partial  \psi_\eps(y_\eps)$.  Lemma~\ref{l.convex+diff} implies that
\begin{align*}
    p = p_\eps +\frac{\eps y_\eps}{\sqrt{1+|y_\eps|^2}}
\end{align*}
for some $p_\eps \in \partial \psi(y_\eps)$.  In particular, we have
\begin{align}\label{e.|p-(p'+n)|}
    |p - p_\eps|\leq \eps.
\end{align}
By Lemma~\ref{l.lift} applied to $p_\eps$, there exists $(b_\eps,p'_\eps)\in[0,\infty)\times \S^K_+$ such that
\begin{gather}
    (b_\eps,p_\eps ) \in \partial f(0,y)\label{e.(b,p'+n)},\\
    p_\eps\leq p'_\eps,\quad p'_\eps\in\S^K_+,\quad |p'_\eps|\leq \|f\|_{\mathrm{Lip}}\label{e.p_eps<p'_eps}\\
    b_\eps -\H(p'_\eps) = 0.\label{e.b-H(p'+n)>0}
\end{gather}

\textit{Step 4.} We are now ready to prove \eqref{e.a-H(p)}. Define $h : \lambda \mapsto f \Ll( \lambda (t,x) + (1-\lambda)(0,y_\eps) \Rr)$ on~$[0,1]$. Clearly, $h$ is convex. By \eqref{e.(b,p'+n)}, the right derivative of $h$ at $0$ satisfies
\begin{equation*}  
h'_+(0) \ge b_\ep t + p_\ep\cdot(x-y_\ep).
\end{equation*}
On the other hand, due to \eqref{e.(a,p)}, the left derivative at $1$ satisfies
\begin{equation*}  
h'_-(1) \le a t + p \cdot(x-y_\ep). 
\end{equation*}
By convexity of $h$, we must have $h'_+(0) \le h'_-(1)$. This along with \eqref{e.|p-(p'+n)|} and \eqref{e.eps*y_to_0} implies that, as $\ep$ tends to zero,
\begin{align*}
    a \ge b_\ep+o(1).
\end{align*}
By \eqref{e.b-H(p'+n)>0}, the definition of $\bar \H$ in \eqref{e.def.bar_H}, and \eqref{e.p_eps<p'_eps}, we have that 
\begin{equation*}  
b_\ep = \H(p_\ep') \ge \bar \H(p_\ep).
\end{equation*}
Using that $\bar \H$ is lower semi-continuous and \eqref{e.|p-(p'+n)|} together with the two previous displays yields that
\begin{equation*}  
a \ge \bar \H(p) + o(1),
\end{equation*}
and \eqref{e.a-H(p)} follows by letting $\ep$ tend to zero. 
%
%
\end{proof}


In the corollary below, we rephrase our criterion for identifying solutions in the following way: instead of asking for the equation to be valid on a dense subset, we ask that it be valid at any point at which the candidate function can be touched from above by a smooth function. As will be seen in the next section, the main advantage to this formulation is that, by convexity, we automatically benefit from a control on the Hessian of the candidate function at the contact point.

\begin{corollary}\label{c.vis}
Let $f : [0,\infty) \times \S^K_+ \to \R$ be nondecreasing, Lipschitz, convex, and have nondecreasing gradients. Suppose that $\psi=f(0,\cdot)$ is $C^1$, and that the following property holds: for every $\phi \in C^\infty((0,\infty)\times\S^K_+)$ and $(t,x) \in (0,\infty)\times\S^K_{++}$ such that $f-\phi$ achieves a strict local maximum at $(t,x)$, we have
\begin{equation*}
    (\partial_t\phi-\H(\nabla\phi))(t,x)=0.
\end{equation*}
Then $f$ is a viscosity solution to \eqref{e.HJ_tensor}.
\end{corollary}
\begin{proof}
Let $\phi$ and $(t,x)$ be as in the statement of the corollary. Since $f$ is convex, we have that, for any $(a,p)\in\partial f(t,x)$ and $(t',x')\in(0,\infty)\times\S^K_+$,
\begin{align*}
    a(t'-t)+p\cdot(x'-x)&\leq f(t',x')-f(t,x)\\
    &\leq\partial_t\phi(t,x)(t'-t)+\nabla \phi(t,x)\cdot (x'-x)+o(|t'-t|+|x'-x|).
\end{align*}
It then follows that $f$ is differentiable at $(t,x)$ and the derivatives of $f$ at $(t,x)$ coincide with those of $\phi$. By Proposition \ref{p.weak-tensor} {and Remark~\ref{r.ae.diff}}, it {therefore} suffices to show that the set
\begin{equation}\label{e.(t,x)_local_max}
    \Big\{(t,x)\in(0,\infty)\times\S^K_{++}:\exists\, \phi\in C^\infty((0,\infty)\times \S^K_{+})\textrm{ s.t. }(t,x)\textrm{ is a local maximum of }f-\phi\Big\}
\end{equation}
is dense {in $[0,\infty)\times \S^K_+$}. (The additional restriction that the local maximum be strict is easily
addressed a posteriori.) {Since the closure of $(0,\infty)\times \S^K_{++}$ is $[0,\infty)\times\S^K_+$, it suffices to show that the set in \eqref{e.(t,x)_local_max} is dense in $(0,\infty)\times \S^K_{++}$.} We fix any $(t,x)\in (0,\infty)\times\S^K_{++}$, and for every $\alpha\geq 1$, we define
\begin{align*}
    \phi_\alpha:\ (t',x')\ \mapsto \  \frac{\alpha}{2}(t'-t)^2+\frac{\alpha}{2}|x'-x|^2.
\end{align*}
Since $f$ is Lipschitz, we can verify that $f-\phi_\alpha$ achieves a global maximum at some point $(t_\alpha, x_\alpha)$.
Using the Lipschitzness of $f$ and that $(f-\phi_\alpha)(t_\alpha,x_\alpha)\geq (f-\phi_\alpha)(t,x)$, we can show that there is a constant $C<\infty$ such that for every $\alpha\geq 1$,
\begin{equation*}
    |t_\alpha-t|+|x_\alpha-x|\leq \frac{C}{\alpha}.
\end{equation*}
This implies that $\lim_{\alpha\to\infty}(t_\alpha,x_\alpha)= (t,x)$. Also, since $(t,x)\in (0,\infty)\times\S^K_{++}$, we have that $(t_\alpha,x_\alpha)\in (0,\infty)\times\S^K_{++}$ for every sufficiently large $\alpha$. Hence $(t_\alpha,x_\alpha)$ belongs to the set in~\eqref{e.(t,x)_local_max}, and we conclude that the set in \eqref{e.(t,x)_local_max} is a dense subset of $[0,\infty)\times \S^K_+$.
\end{proof}

\section{Convergence and application}\label{s.proof}

The main goal of this section is to prove Theorem~\ref{main_thm}, using the tools developed in the previous section. For illustration, we also apply the theorem to a specific model.

\subsection{Convergence}

In view of Proposition~\ref{p.hopf}, Theorem~\ref{main_thm} follows from the next theorem.

\begin{theorem}\label{t.cvg}
Under the conditions of Theorem~\ref{main_thm}, the function $\bar F_N$ converges pointwise to the unique viscosity solution to \eqref{e.HJ_tensor} with initial condition $\psi$. 
\end{theorem}
In order to prove this result, we start by recalling from \cite[Proposition~3.1]{HBJ} (cf.\ also \cite[Proposition~1.2]{mourrat2019hamilton}) that the function $\bar F_N$ satisfies an approximate form of the equation. In \eqref{e.approx.hj}, we implicitly understand that the relevant functions are evaluated at $(t,h) \in [0,\infty) \times \S^K_+$.
\begin{proposition}[Approximate Hamilton-Jacobi equation]
\label{p.approx.hj}
There exists $C < \infty$ such that for every $N \ge 1$ and uniformly over $[0,\infty) \times \S^K_+$, 
\begin{equation}  
\label{e.approx.hj}
\big|\partial_t\bar F_N-\H(\nabla\bar F_N)\big|^2\leq C\kappa(h)N^{-\frac{1}{4}}\big(\Delta \bar F_N+|h^{-1}|\big)^{\frac{1}{4}}+C\E\Ll[\big|\nabla F_N-\nabla \bar F_N\big|^2\Rr],
\end{equation}
where $\kappa$ is the condition number of $h\in \S^K_+$ given by
\begin{equation*}
\kappa(h) := 
\left\{
\begin{array}{ll}
|h||h^{-1}|,    &\quad\text{if }h\in \S^K_{++},  \\
+\infty     &\quad\text{otherwise}.
\end{array}
\right.
\end{equation*}
\end{proposition}

\begin{proof}[Proof of Theorem \ref{t.cvg}]
Since $\bar F_N$ is Lipschitz uniformly in $N$ by Lemma~\ref{l.lip_unif_N}, the Arzel\'a-Ascoli theorem implies that, for every subsequence of $(\bar F_N)_{n\in\N}$, there is a further subsequence converging to some function $f$ in the local uniform topology. It suffices to show that $f$ is a viscosity solution to \eqref{e.HJ_tensor} and the uniqueness is ensured by Proposition~\ref{p.comp}. For convenience, we assume that the whole sequence $(\bar F_N)_{N\in\N}$ converges to $f$.

Lemmas~\ref{l.lip_unif_N} and \ref{l.F_N.conv} {(see also \eqref{e.bar_F_nondec})} ensure that $f$ is nondecreasing, Lipschitz and convex. Since $\bar F_N$ and $f$ are convex, we have
\begin{align*}
\lim_{N\to\infty}(\partial_t, \nabla)\bar F_N(t,h)    = (\partial_t, \nabla)f(t,h)
\end{align*}
at every differentiable point $(t,h)$ of $f$ (indeed, any limit point of $(\partial_t, \nabla)\bar F_N(t,h)$ must belong to the subdifferential of $f$ at $(t,h)$, which is a singleton if $f$ is differentiable at~$(t,h)$). This along with Lemma~\ref{l.F_N.order} yields that $f$ has nondecreasing gradients.
Let $(t,h)\in(0,\infty)\times \S^K_{++}$ and $\phi\in C^\infty((0,\infty)\times\S^K_+)$ be such that $f-\phi$ has a strict local maximum at $(t,h)$. By Corollary~\ref{c.vis}, it suffices to show that 
\begin{equation}
\label{e.eq.phi}
 (\partial_t\phi-\H(\nabla\phi))(t,h)=0.
\end{equation}
Since $\bar F_N$ converges locally uniformly to $f$, there exists $(t_N,h_N) \in [0,\infty) \times \S^K_+$ such that $\bar F_N - \phi$ has a local maximum at $(t_N,h_N)$, and $(t_N,h_N)$ converges to $(t,h)$ as $N$ tends to infinity. Since $(t,h) \in (0,\infty) \times \S^K_{++}$, each $(t_N,h_N)$ also ultimately belongs to $(0,\infty) \times \S^K_{++}$, and without loss of generality, we can assume that every $(t_N,h_N)$ remains a positive distance away from the boundary of $[0,\infty) \times \S^K_+$, uniformly over $N$. Notice that 
\begin{equation}  
\label{e.gradient.equalities}
(\dr_t \bar F_N - \dr_t \phi)(t_N,h_N) = 0 \quad \text{ and } \quad (\nabla \bar F_N - \nabla \phi)(t_N,h_N) = 0.
\end{equation}
Throughout the rest of the proof, we use the letter $C < \infty$ to denote a constant whose value may change from one occurrence to the next, and is allowed to depend on $(t,h)$ and~$\phi$. We decompose the argument into three steps.

\textit{Step 1.} We show that for every $h' \in \S^K_+$ with $|h'| \le C^{-1}$, we have
\begin{equation}
\label{e.hessian.est}
0 \le \bar F_N(t_N,h_N + h') - \bar F_N(t_N,h_N ) - h' \cdot \nabla \bar F_N(t_N,h_N) \le C |h'|^2.
\end{equation}
The first inequality follows from the convexity of $\bar F_N$. To derive the second inequality, we start by writing Taylor's formula:
\begin{multline}  
\label{e.taylor}
\bar F_N(t_N,h_N + h') - \bar F_N(t_N,h_N ) 
\\= h' \cdot \nabla \bar F_N(t_N,h_N) + \int_0^1 (1-s) h'\cdot \nabla\Ll(h' \cdot \nabla \bar F_N\Rr)(t_N,h_N + s h') \, \d s.
\end{multline}
The same formula also holds if we substitute $\bar F_N$ by $\phi$ throughout. Since $\bar F_N - \phi$ has a local maximum at $(t_N,h_N)$, we have for every $|h'| \le C^{-1}$ that
\begin{equation*}  
\bar F_N(t_N,h_N + h') - \bar F_N(t_N,h_N ) \le \phi(t_N,h_N + h') - \phi(t_N,h_N ) .
\end{equation*}
Using also \eqref{e.gradient.equalities}, we obtain that
\begin{equation*}  
\int_0^1 (1-s)  h'\cdot \nabla\Ll(h' \cdot \nabla \bar F_N\Rr)(t_N,h_N + s h')  \, \d s \le \int_0^1 (1-s)  h'\cdot \nabla\Ll(h' \cdot \nabla \phi\Rr)(t_N,h_N + s h')  \, \d s.
\end{equation*}
Since the function $\phi$ is smooth, the right side of this inequality is bounded by $C |h'|^2$. Using \eqref{e.taylor} once more, we obtain \eqref{e.hessian.est}. 

\textit{Step 2.} 
Let
\begin{equation*}  
D := \Ll\{ (t',h') \in [0,\infty) \times \S^K_+ \ : \ |t'-t| \le C^{-1} \text{ and } |h'-h| \le C^{-1} \Rr\} .
\end{equation*}
In this step, we show that 
\begin{equation}
\label{e.concentration}
\E \Ll[ |\nabla F_N - \nabla \bar F_N|^2(t_N,h_N) \Rr] \le C \Ll(\E \Ll[ \sup_D |F_N - \bar F_N|^2  \Rr]\Rr)^\frac 1 2.
\end{equation}
We recall from \cite[(3.13)]{HBJ} that, for every $a \in \S^K$ and $(t',h') \in [0,\infty) \times \S^K_+$ such that $|h'-h| \le C^{-1}$, we have
\begin{equation*}  
a \cdot \nabla (a \cdot \nabla F_N)(t',h') \ge -C |a|^2 \frac{|Z|}{\sqrt{N}},
\end{equation*}
and that {$Z\in\R^{N\times K}$ is the matrix of independent standard Gaussians appearing in the definition of $\overline{Y}$ (see the second paragraph in Section~\ref{s.setting}).} {Applying Taylor's formula as in Step $1$, it is readily verified that} for every $|h'| \le C^{-1}$, we have
\begin{equation*}  
F_N(t_N,h_N+ h') \ge F_N(t_N,h_N) + h' \cdot \nabla F_N(t_N,h_N) - C |h'|^2 \frac{|Z|}{\sqrt{N}}.
\end{equation*}
Combining this with \eqref{e.hessian.est}, we obtain that, for every $|h'| \le C^{-1}$,
\begin{equation*}  
h' \cdot \Ll(\nabla F_N - \nabla \bar F_N\Rr)(t_N,h_N)\le 2 \sup_D |F_N - \bar F_N| + C |h'|^2 \Ll(1+\frac{|Z|}{\sqrt{N}}\Rr).
\end{equation*}
For some deterministic $\lambda \in [0,C^{-1}]$ to be determined, we fix {the random matrix}
\begin{equation*}  
h' := \lambda \frac{\Ll(\nabla F_N - \nabla \bar F_N\Rr)(t_N,h_N)}{|\Ll(\nabla F_N - \nabla \bar F_N\Rr)(t_N,h_N)|},
\end{equation*}
so that 
\begin{equation*}  
\lambda |\nabla F_N - \nabla \bar F_N|(t_N,h_N)  \le 2 \sup_D |F_N - \bar F_N|  + C \lambda^2 \Ll(1+\frac{|Z|}{\sqrt{N}}\Rr).
\end{equation*}
Squaring this expression and taking the expectation yields
\begin{equation*}  
\lambda^2 \E \Ll[ |\nabla F_N - \nabla \bar F_N|^2(t_N,h_N) \Rr] \le 8 \E \Ll[ \sup_D |F_N - \bar F_N|^2  \Rr]  + C \lambda^4 \E \Ll[  \Ll(1+\frac{|Z|}{\sqrt{N}}\Rr)^2\Rr] .
\end{equation*}
Since $\E[|Z|^2] = NK$, choosing $\lambda^4 = \E \Ll[ \sup_D |F_N - \bar F_N|^2  \Rr] $ yields \eqref{e.concentration}.

\textit{Step 3.} Recall that we assume that $\E \Ll[ \sup_D |F_N - \bar F_N|^2  \Rr] $ tends to zero as $N$ tends to infinity. By Proposition~\ref{p.approx.hj}, \eqref{e.hessian.est}, and \eqref{e.concentration}, we obtain that
\begin{equation*}  
\lim_{N \to \infty} \Ll(\partial_t\bar F_N-\H(\nabla\bar F_N)\Rr)(t_N,h_N) = 0.
\end{equation*}
Using also \eqref{e.gradient.equalities} and the fact that the function $\phi$ is smooth, this yields \eqref{e.eq.phi}, and thus completes the proof.
\end{proof}

\subsection{Application}\label{s.app}
We study the model considered in \cite{luneau2019mutual}, which corresponds to \eqref{e.Y} with $L=1$, $\mathsf{p}\in\N$, and $A\in\R^{K^\mathsf{p}\times 1}$ given by $A_{\mathbf j}=1$ if $j_1=j_2=\cdots = j_\mathsf{p}$ and zero otherwise. Here, we used the multi-index notation $\mathbf j=(j_1,j_2,\dots,j_\mathsf{p})\in \{1,\dots,K\}^\mathsf{p}$. Explicitly, this model can be expressed as
\begin{align}\label{e.special}
    Y_{\mathbf i} = \sqrt{\frac{2t}{N^{\mathsf{p}-1}}} \sum_{j=1}^K\prod_{n=1}^\mathsf{p}X_{i_n, j}+W_{\mathbf i},\quad \mathbf i \in \{1,\cdots, N\}^\mathsf{p},
\end{align}
where $X\in \R^{N\times K}$ is assumed to have i.i.d.\ row vectors {with norms bounded by $\sqrt{K}$ almost surely. Hence, the condition in \eqref{e.supp} is satisfied.} For even $\mathsf{p}$, the limit of the free energy associated with this model has been proved to satisfy a variational formula in \cite{luneau2019mutual}. When $\mathsf{p}$ is odd, the situation is more difficult; in  \cite{HBJ}, it was only proven that the limit is bounded above by a variational formula. Here, we will apply Theorem~\ref{main_thm} to treat both even and odd values of $\mathsf{p}$.

Recall the definition of $\H$ in \eqref{def.H}. In this case, the nonlinearity $\H$ is given by
\begin{align}\label{e.H_sp}
    \H(q) = \sum_{k,k'=1}^K(q_{k,k'})^\mathsf{p},\quad\forall q\in \S^K_+.
\end{align}
Since row vectors of $X$ are i.i.d., we have $\bar F_N(0,\cdot)= \bar F_1(0,\cdot)$ for all $N\in\N$. Setting $\psi :=  \bar F_1(0,\cdot)$ and using the formula for $F_N$ in \eqref{e.F_N}, we have
\begin{align}\label{e.psi_sp}
    \psi(h) = \E \log \int_{\R^{1\times K}} \exp\bigg(2h\cdot (x^\intercal X_{1,\cdot})+\sqrt{2h}\cdot(x^\intercal Z)-h\cdot(x^\intercal x) \bigg) \d P(x),\quad\forall h\in\S^K_+,
\end{align}
where $P$ is the law of the first row vector $X_{1,\cdot} = (X_{1,k})_{1\leq k\leq K}$. By Lemma~\ref{l.lip_unif_N}, $\psi$ is $C^1$.
The concentration condition $\lim_{N\to\infty}\E\|F_N-\bar F_N\|^2_{L^\infty(D)}=0$ for each compact $D\subset [0,\infty)\times \S^K_+$ is proved in \cite[Lemma C.1]{HBJ}. Hence, the next result follows from Theorem~\ref{main_thm}.
\begin{corollary}
Under the assumption \eqref{e.supp}, in the model described above with $\mathsf{p}\in\N$, it holds that, for every $(t,h)\in[0,\infty)\times\S^K_+$,
\begin{align*}
    \lim_{N\to\infty}\bar F_N(t,h) = \sup_{h''\in\S^K_+}\inf_{h'\in\S^K_+}\big\{h''\cdot (h-h')+\psi(h')+t\H(h'')\big\},
\end{align*}
for $\H$ and $\psi$ given in \eqref{e.H_sp} and \eqref{e.psi_sp}, respectively.
\end{corollary}

\subsection{{Simplification of the variational formula}}

We describe a way of simplifying the formula~\eqref{e.main_thm} under the additional assumption that the mapping $\H$ in \eqref{def.H} only depends on the diagonal entries of its argument. 

We introduce the linear map $\diag : \R^K\to \S^K$ defined by $\diag x = \mathrm{diag}(x_1,\dots,x_K)$. Its adjoint $\diag^*:\S^K\to\R^K$ is given by $\diag^* h = (h_{11},\dots,h_{KK})$ for $h\in \S^K$. Note that $\diag^*\diag$ is the identity map on $\R^K$, and $\diag\diag^*h = \mathrm{diag}(h_{11},\ldots,h_{KK})$ for every $h\in \S^K$. The additional assumption on $\H$ can be reformulated as
\begin{align}\label{e.assump_H_diag}
    \H(q) = \H(\diag\diag^*q),\quad\forall q\in \S^K_+.
\end{align}
For $x,x'\in\R^K$, we write $x\cdot x' = \sum_{i=1}^Kx_ix'_i$. We set $\R_+^K= [0,\infty)^K$. Note that $\diag (\R^K_+)$ contains exactly the diagonal matrices in $\S^K_+$. For $F_N$ given in \eqref{e.F_N}, we want to show that, under the assumptions of Theorem~\ref{main_thm} and for every $t \ge 0$ and $ x \in \R_+^K$, we have
\begin{align}
\label{e.main_thm_simpl}
    \lim_{N\to\infty}\bar F_N(t,\diag x) = \sup_{x''\in\R^K_+}\inf_{x'\in\R^K_+}\big\{x''\cdot \left(x-x'\right)+\psi\left(\diag x'\right)+t\H\left(\diag x''\right)\big\}.
\end{align}
In particular, setting $x=0$, we obtain a simpler representation of the limit of the original free energy $F^\circ_N$.

The proof of this statement can be achieved by working with the following Hamilton--Jacobi equation:
\begin{align}\label{e.g_hj_diag}
    \partial_t g - \H(\diag \nabla g) =0,\quad \text{on }[0,\infty)\times\R^K_+.
\end{align}
The well-posedness of this equation and the representation of the solution by the Hopf formula can be established in a similar way (see \cite[Section~2]{chen2022hamilton}). A corresponding identification criterion for solutions, as stated in Proposition~\ref{p.weak-tensor}, can also be obtained. There, the partial order defining the notion of nondecreasingness, as in \eqref{e.partial_oder_1} and \eqref{e.partial_oder_2}, is now induced by the cone $\R^K_+$. Lastly, for any differentiable function $\phi:\S^K_+\to\R$, we can verify that, 
\begin{align*}
    \nabla \phi^\diag( x) = \diag^* \nabla \phi(h)\big|_{h = \diag x},\quad\forall x \in \R^K_+,
\end{align*}
where $\phi^\diag:\R^K_+\to\R$ is given by $\phi^\diag = \phi(\diag\,\cdot\,)$. Hence, setting $\bar F^\diag_N(t,x) = \bar F_N(t,\diag x)$, and using Proposition~\ref{p.approx.hj} and \eqref{e.assump_H_diag}, we can see that $\bar F^\diag_N$ approximately solves \eqref{e.g_hj_diag} and that a similar estimate in Proposition~\ref{p.approx.hj} holds for $\bar F^\diag_N$. Then, the same argument as in the proof of Theorem~\ref{t.cvg} yields that $\bar F^\diag_N$ converges to the unique viscosity solution of \eqref{e.g_hj_diag} with initial condition $\psi(\diag\,\cdot\,)$. Due to the convexity of $\psi(\diag\,\cdot\,)$, the solution admits a representation by the Hopf formula, which is exactly the right-hand side in \eqref{e.main_thm_simpl}.

As a concrete example inspired by \cite{accm}, suppose as in the previous subsection that $X \in \R^{N\times K}$ has i.i.d.\ row vectors with norm bounded by $\sqrt{K}$, but this time we observe, for each $i,j \in \{1,\ldots, N\}$ and $k \in \{1,\ldots, K-1\}$, the quantity
\begin{equation*}  
\sqrt{\frac{2t}{N}} X_{i,k} X_{j,k+1} + W_{i,j,k},
\end{equation*}
where $(W_{i,j,k})_{i,j \le N, k < K}$ are independent standard Gaussians, independent of $X$. This can be mapped into our setting by choosing $\mathsf{p}=2$, $L = K-1$, $A \in \R^{K^2 \times (K-1)}$ given by $A_{(k,l),r} = 1$ if $r = k = l-1$ and zero otherwise. With this choice of $A$, the function $\H$ takes the form
\begin{equation*}  
\H(q) = \sum_{k = 1}^{K-1} q_{k,k} \, q_{k+1,k+1} = \H(\diag\diag^* q), \qquad \forall q \in \S^K_+.
\end{equation*}
We thus obtain that the limit free energy $F_N^\circ(t) = F_N(t,0)$ is given by
\begin{equation}  
\label{e.formula.physicists}
\lim_{N \to \infty} F_N^\circ(t) = \sup_{x'\in\R^K_+}\inf_{x\in\R^K_+}\Ll\{\psi^\diag\!\left(x\right)- x \cdot x' +t \sum_{k = 1}^{K-1} x'_k x'_{k+1}\Rr\}.
\end{equation}
Moreover, under the additional assumption that the coordinates of the vector $(X_{1,k})_{1 \le k \le K}$ are independent, the initial condition $\psi^\diag$ can be decomposed into a sum of functions of one variable: there exist convex and nondecreasing functions $\psi_1,\ldots, \psi_K : \R_+ \to \R$ such that for every $x \in \R_+^K$,
\begin{equation*}  
\psi^\diag\!(x) = \sum_{k = 1}^K \psi_k(x_k). 
\end{equation*}
(Cases in which different layers have different lengths, say for instance $X_{i,k} = 0$ for every $i > \al_k N$ for some fixed $\al_k \in (0,1)$, can be covered as well, and this translates into multiplying each $\psi_k$ by a suitable scalar.) Under these conditions, the formula~\eqref{e.formula.physicists} can be further simplified, as we now explain. For each $x \in \R^K$, we denote by $x_o = (x_1,x_3,\ldots, x_{2\cdot\lfloor (K-1)/2 \rfloor + 1})$ and $x_e = (x_2, x_4,\ldots, x_{2 \cdot \lfloor K/2 \rfloor})$ respectively the odd and even coordinates of the vector $x$, and for each $k \in \{1,\ldots, K\}$ and $y \ge 0$, we set
\begin{equation*}  
\psi_k^*(y) := \sup_{x\ge 0} (xy - \psi_k(x)).
\end{equation*}
By \cite[Theorem~12.4]{rockafellar1970convex}, we have that $\psi_k^{**} = \psi_k$. Moreover, we can write
\begin{equation*}  
\inf_{x_e} \Ll\{ \sum_{k = 1}^K \psi_k(x_k) - x \cdot x' + t \sum_{k = 1}^{K-1} x'_k x'_{k+1}\Rr\} = \sum_{k \text{ odd}} \Ll(\psi_k(x_k) -x_k x_k'\Rr) - \sum_{k \text{ even}} \psi^*_k(x_k') + t \sum_{k = 1}^{K-1} x'_k x'_{k+1},
\end{equation*}
and observe that the optimization problems over $x_o$ and $x_e'$ are separated. We can thus interchange $\sup_{x'_e}$ and $\inf_{x_o}$ to get that 
\begin{align*}  
\lim_{N \to \infty} F_N^\circ(t) & = \sup_{x'_o} \inf_{x_o} \sup_{x'_e} \Ll\{ \sum_{k \text{ odd}} \Ll(\psi_k(x_k) -x_k x_k'\Rr) - \sum_{k \text{ even}} \psi_k^*(x_k') + t \sum_{k = 1}^{K-1} x'_k x'_{k+1} \Rr\} 
\\
& = \sup_{x'_o} \inf_{x_o} \Ll\{\sum_{k \text{ odd}} \Ll( \psi_k(x_k)- x_k x_k'\Rr) + \sum_{k \text{ even}} \psi_k\Ll(t x_{k-1}' +  t x_{k+1}'\Rr)\Rr\},
\end{align*}
using that $\psi_k^{**} = \psi_k$, and with the understanding that $x_{K+1} = 0$. 
Similar formulas were first obtained in \cite{accm}.

\appendix

\section{On convex viscosity solutions}\label{a.vis}

The goal of this section is to demonstrate the workings of a convenient uniqueness criterion for Hamilton-Jacobi equations, in the simpler context of equations posed on $[0,\infty) \times \R^d$. This criterion states that, if the function under consideration is \textit{convex}, then we can assert that it is the viscosity solution of some Hamilton-Jacobi equation as soon as it satisfies the equation on a dense subset and the initial condition is of class $C^1$. This criterion is generalized to equations posed on $[0,\infty) \times \S^K_+$ in Proposition~\ref{p.weak-tensor}.

Let $\msf H: \R^d \to \R$ be a smooth function. We start by recalling the notion of viscosity solutions to
\begin{equation}  
\label{e.hj}
\dr_t f - \msf H(\nabla f) = 0 \qquad \text{on } \ [0,\infty) \times \R^d.
\end{equation}

\bigskip
\bigskip
\bigskip
\bigskip
\bigskip
\bigskip
\bigskip
\bigskip
\bigskip

\begin{definition}
\leavevmode
\begin{enumerate}
    \item A continuous function $f:[0,\infty)\times \R^d\to \R$ is a viscosity subsolution to \eqref{e.HJ_tensor} if for every $(t,h) \in (0,\infty)\times \R^d $ and every smooth $\phi:(0,\infty)\times \R^d\to \R$ such that $f-\phi$ has a local maximum at $(t,h)$, we have
\begin{align*}
    \big(\partial_t \phi - \H(\nabla\phi)\big)(t,h)\leq 0.
\end{align*}

\item  A continuous function $f:[0,\infty)\times \R^d\to \R$ is a viscosity supersolution to \eqref{e.HJ_tensor} if for every $(t,h) \in (0,\infty)\times \R^d$ and every smooth $\phi:(0,\infty)\times \R^d \to \R$ such that $f-\phi$ has a local minimum at $(t,h)$, we have
\begin{align*}
    \big(\partial_t \phi - \H(\nabla\phi)\big)(t,h)\geq 0.
\end{align*}

\item A continuous function $f:[0,\infty)\times \R^d\to \R$ is a viscosity solution to \eqref{e.HJ_tensor} if $f$ is both a viscosity subsolution and supersolution.

\end{enumerate}
\end{definition}

The main goal of this section is to prove the following proposition.
\begin{proposition}
\label{p.weak}
Let $f : [0,\infty) \times \R^d \to \R$ be  Lipschitz and convex. Suppose that $f$ satisfies \eqref{e.hj} on a dense subset of $(0,\infty) \times \Rd$, and that the initial condition $f(0,\cdot)$ is~$C^1$. Under these conditions, the function $f$ is a viscosity solution to \eqref{e.hj} with initial condition $f(0,\cdot)$. 
\end{proposition}

\begin{remark}  
\label{r.diff.assumption}
In Proposition~\ref{p.weak}, the assumption that $f(0,\cdot)$ be $C^1$ is necessary. Indeed, notice for instance that 
\begin{equation*}  
f(t,x) := |x| - t
\end{equation*}
is convex and satisfies
\begin{equation*}  
\dr_t f + |\nabla f|^2 = 0
\end{equation*}
at every point of differentiability of $f$. However, since the null function is clearly a solution, the statement that $f$ is also a solution would contradict the maximum principle. Instead, the viscosity solution to this equation with same initial condition is given by the Hopf-Lax formula
\begin{equation*}  
(t,x) \mapsto \inf_{y \in \R} \Ll( |y| + \frac{|y-x|^2}{4t} \Rr) = 
\Ll\{
\begin{array}{ll}  
\frac{|x|^2}{4t} & \text{if } |x| \le 2t, \\
|x| - t & \text{if } |x| > 2t.
\end{array}
\Rr.
\end{equation*}
\end{remark}
\begin{proof}[Proof of Proposition~\ref{p.weak}]
Recall the definition of subdifferential in \eqref{e.def.subdiff}. We decompose the proof into three steps. 

\textit{Step 1.} 
We check that $f$ must be a subsolution to \eqref{e.hj}. Let $\phi \in C^\infty((0,\infty) \times \R^d)$, and $(t,x) \in (0,\infty) \times \R^d$ be such that $f-\phi$ has a  local maximum at $(t,x)$. We then have
\begin{equation*}  
f(t',x') - f(t,x) \le (t'-t) \dr_t \phi(t,x) + (x'-x)\cdot  \nabla \phi(t,x) + o(|t'-t| + |x'-x|).
\end{equation*}
This {along with the convexity of $f$} implies that the subdifferential $\dr f(t,x)$ is the singleton $\{(\dr_t \phi,\nabla  \phi)(t,x)\}$, and thus that $f$ is differentiable at $(t,x)$, with $(\dr_tf ,\nabla f )(t,x) = (\dr_t\phi,\nabla\phi)(t,x)$. Using similar arguments as in Remark~\ref{r.ae.diff}, we deduce that 
\begin{equation*}  
\Ll(\dr_t \phi - \msf H(\nabla \phi)\Rr)(t,x) = \Ll(\dr_t f - \msf H(\nabla f)\Rr)(t,x) = 0,
\end{equation*}
as desired.

\textit{Step 2.} In the next two steps, we show that $f$ is a supersolution to \eqref{e.hj}. Let  $(a,p) \in \dr f(t,x)$. In view of Lemma~\ref{l.subsol_subd}, it is sufficient to show that
\begin{equation}
\label{e.cond.super}
a - \msf H(p) \ge 0.
\end{equation}
Since $(a,p) \in \dr f(t,x)$ {and $f$ is convex}, we have for every $(t',x') \in [0,\infty) \times \R^d$ that
\begin{equation*}  
f(t',x') \ge f(t,x) +  (t'-t) a + (x'-x) \cdot p.
\end{equation*}
In particular, the mapping $y \mapsto f(0,y) - y \cdot p$ is bounded from below. In this step, we assume that the infimum 
\begin{equation}  
\label{e.infimum}
\inf_{y \in \R^d} \Ll(f(0,y) - y \cdot p\Rr)
\end{equation}
is achieved, and we denote by $y$ a point realizing the infimum. By arguing as in Remark~\ref{r.ae.diff}, we see that there exists $(b,p') \in \dr f(0,y)$ such that $b - \msf H(p') = 0$. Since $f(0,\cdot)$ is $C^1$ and $y$ realizes \eqref{e.infimum}, we must have $p' = \dr_y f(0,y) = p$, and thus $(b,p) \in \dr f(0,y)$ with $b - \msf H(p) = 0$.

{Due to the convexity of $f$,}  the mapping $g : \lambda \mapsto f \Ll( \lambda (t,x) + (1-\lambda)(0,y) \Rr)$ is convex over the interval $[0,1]$. Since $(b,p) \in \dr f(0,y)$, the right derivative of $g$ at $0$ satisfies
\begin{equation*}  
g'_+(0) \ge b t + p\cdot(x-y).
\end{equation*}
{Due to $(a,p)\in \partial f(t,x)$, the left derivative at $1$ satisfies}
\begin{equation*}  
g'_-(1) \le a t + p \cdot(x-y). 
\end{equation*}
By convexity of $g$, we must have $g'_+(0) \le g'_-(1)$, and thus $a \ge b$. Recalling that $b - \msf H(p) = 0$, we obtain \eqref{e.cond.super}, as desired. 

\textit{Step 3.} To conclude, there remains to consider the case when the infimum in \eqref{e.infimum} is not achieved. For every $\ep > 0$, we consider 
\begin{equation*}  
\inf_{y \in \Rd} \Ll(f(0,y)  + \ep |y| - y \cdot p\Rr).
\end{equation*}
This infimum is achieved at a point $y_\ep \in \Rd$, and
\begin{equation}  
\label{e.grad.close.to.p}
\Ll|\nabla f(0,y_\ep) - p \Rr| \le \ep.
\end{equation}
Moreover, 
\begin{equation*}  
\limsup_{\ep \to 0} \inf_{y \in \Rd}  \Ll(f(0,y)  + \ep |y| - y \cdot p\Rr) = \inf_{y \in \Rd} \Ll(f(0,y)  - y \cdot p\Rr),
\end{equation*}
and
\begin{equation*}  
f(0,y_\ep) - y_\ep \cdot p \ge \inf_{y \in \Rd} \Ll(f(0,y)  - y \cdot p\Rr),
\end{equation*}
so that
\begin{equation}
\label{e.lim.yep}
\lim_{\ep \to 0} \ep |y_\ep| = 0.
\end{equation}
Following the argument in Step 2, we can find $b_\ep \in \R$ such that $(b_\ep,\nabla f(0,y_\ep)) \in \dr f(0,y_\ep)$ and $b_\ep - \msf H(\nabla f(0,y_\ep)) = 0$. Continuing as in Step 2, we then obtain that 
\begin{equation*}  
b_\ep t + \nabla f(0,y_\ep) \cdot (x-y_\ep) \le a t + p \cdot(x-y_\ep).
\end{equation*}
Using \eqref{e.grad.close.to.p} and \eqref{e.lim.yep}, we deduce that, as $\ep$ tends to zero,
\begin{equation*}  
a \ge b_\ep + o(1). 
\end{equation*}
Recalling that $b_\ep  - \msf H(\nabla f (0,y_\ep)) = 0$, and using again \eqref{e.grad.close.to.p} and the continuity of $\msf H$, we obtain \eqref{e.cond.super}.
\end{proof}

\small
\bibliographystyle{abbrv}
\bibliography{tensors}

\end{document}